\documentclass[11pt]{article}
\usepackage[margin=1in]{geometry}
\usepackage{amsfonts,amssymb,amsmath}
\usepackage{epstopdf}
\usepackage{algorithm}
\usepackage{algorithmic}
\usepackage{multirow}
\usepackage{color}
\usepackage{float} 
\usepackage{graphicx}
\usepackage{subcaption}
\usepackage{amsthm}
\newtheorem{theorem}{Theorem}[section]
\newtheorem{corollary}[theorem]{Corollary}
\newtheorem{lemma}[theorem]{Lemma}
\newtheorem{proposition}[theorem]{Proposition}
\newtheorem{remark}[theorem]{Remark}

\begin{document}

\title
{Provably Convergent Plug-and-play Proximal Block Coordinate Descent Method for Hyperspectral Anomaly Detection}

\author{Xiaoxia Liu\thanks{School of Mathematics, South China University of Technology, Guangzhou, Guangdong, 510641, China (Email: xiaoxia\_liu\_math@outlook.com).}  \quad \quad
		Shijie Yu\thanks{School of Mathematical Sciences, Shenzhen University, Shenzhen 518060, China; Department of Applied Mathematics, The Hong Kong Polytechnic University, Hong Kong SAR, China (Email: yushijie0404@163.com).
		}}

\date{}

\maketitle

\begin{abstract}
Hyperspectral anomaly detection refers to identifying pixels in the hyperspectral images that have spectral characteristics significantly different from the background. In this paper, we introduce a novel model that represents the background information using a low-rank representation. We integrate an implicit proximal denoiser prior, associated with a deep learning based denoiser, within a plug-and-play (PnP) framework to effectively remove noise from the eigenimages linked to the low-rank representation. Anomalies are characterized using a generalized group sparsity measure, denoted as $\|\cdot\|_{2,\psi}$. To solve the resulting orthogonal constrained nonconvex nonsmooth optimization problem, we develop a PnP-proximal block coordinate descent (PnP-PBCD) method, where the eigenimages are updated using a proximal denoiser within the PnP framework. We prove that any accumulation point of the sequence generated by the PnP-PBCD method is a stationary point.  We evaluate the effectiveness of the PnP-PBCD method on hyperspectral anomaly detection in scenarios with and without Gaussian noise contamination. The results demonstrate that the proposed method can effectively detect anomalous objects, outperforming the competing methods that may mistakenly identify noise as anomalies or misidentify the anomalous objects due to noise interference.
\end{abstract}

\noindent{\bf Keywords}:
Low-rank representation, proximal block coordinate descent, hyperspectral anomaly detection, plug-and-play.

\noindent{\bf  2000 Mathematics Subject Classification}:
L15A18, 68U10, 46N10.

\maketitle

\section{Introduction}\label{intro}

Hyperspectral anomaly detection aims to identify pixels or regions in hyperspectral images~(HSIs) that significantly differ from the surrounding background without prior knowledge of the target spectral information. These pixels, often referred to as anomalies, could represent objects or materials such as aircraft, ships, vehicles, or other structures that deviate from the natural background. Detecting such anomalies is crucial due to their significance in various applications. For example, in environmental monitoring, anomalies may indicate areas affected by pollution or disease in vegetation \cite{pour2021editorial}; in the food industry, anomalies may be detected for quality control by identifying physical defects and inconsistencies in products \cite{wu2013advanced}. By leveraging the rich spectral information provided by HSIs, the accuracy and reliability of anomaly detection can be enhanced, thereby improving decision-making processes in fields such as security, agriculture, and resource management.

In hyperspectral anomaly detection, the Reed-Xiaoli (RX) method, introduced by Reed and Xiaoli in 1990 \cite{reed1990adaptive}, is a foundational method known for its simplicity and widespread adoption. The RX method assumes that background spectral features follow a multivariate Gaussian distribution and identifies anomalies by calculating the Mahalanobis distance from the background. Over time, RX has inspired several variants to address its limitations in real-world applications. For example, the local RX method~\cite{molero2013analysis} enhances localized anomaly detection using sliding windows for background estimation; the kernel RX method~\cite{kwon2005kernel} maps data into high-dimensional feature spaces to better adapt to nonlinear distributions; and the weighted RX method~\cite{guo2014weighted} introduces pixel-level weighting for improving robustness against noise. While RX and its variants are computationally efficient and serve as benchmarks in the field, they often rely on Gaussian assumptions and are sensitive to noise and outliers, limiting their performance in complex scenes.

In contrast to statistical approaches like RX, representation-based methods focus on explicitly modeling the structure of HSIs without assuming a predefined distribution. Li et al.~\cite{li2015hyperspectral} proposed the background joint sparse representation detection (BJSRD) method, which reconstructs each background pixel using a sparse set of coefficients from a dictionary. Xu et al.~\cite{xu2015anomaly} introduced the low-rank and sparse representation (LRASR) method, which models the background as a low-rank component while representing anomalies as sparse components. Feng et al. \cite{feng2021local} developed the local spatial constraint and total variation~(LSC-TV) method, which combines low-rank modeling with superpixel segmentation and total variation~(TV) regularization to effectively separate anomalies in complex scenes. To preserve the intrinsic 3D structure of HSIs, the low-rank component is characterized using tenor low-rank representation.  For example, the tensor low-rank and sparse representation (TLRSR) method~\cite{wang2022learning} utilizes the tensor singular value decomposition~(t-SVD), while the method proposed in \cite{feng2023hyperspectral} employs the tensor ring decomposition.

Deep learning methods have significantly improved hyperspectral anomaly detection by extracting hierarchical features from high-dimensional data using deep neural networks. Among these, the Auto-AD method~\cite{wang2021auto}, a fully convolutional autoencoder, autonomously reconstructs the background and highlights anomalies through reconstruction errors, eliminating the need for manual parameter tuning or preprocessing. Other neural network models, such as stacked denoising autoencoders (SDAs)~\cite{zhao2018spectral} and spectral-constrained adversarial autoencoders (SC-AAE) \cite{xie2019spectral}, use manifold learning and adversarial strategies to enhance anomaly detection capabilities. These approaches are highly effective in nonlinear and complex environments but often require large datasets and significant computational resources, which can pose challenges for real-time applications.

In this paper, we develop a novel approach for hyperspectral anomaly detection that utilizes a representation-based technique for expressing the background, a deep learning denoiser for reducing noise contamination and a group sparsity measure for identifying anomalies. Our main contributions are summarized as follows:

\begin{itemize}
    \item We represent the background of HSIs in terms of a tensor mode-$3$ product of a learnable orthogonal basis as the subspace and a tensor formed by eigenimages as the representation coefficients.

    \item  We employ a deep learning denoiser in a plug-and-play~(PnP) fashion to eliminate the noise from the eigenimages. We enhance the existing relaxed proximal denoiser to its shifted version to denoise the eigenimages that may not fall within the pretrained range. The proposed denoiser can also be viewed as a proximal operator associated with a weakly convex function.
    
    \item We introduce a generalized group sparsity measure, $\|\cdot\|_{2,\psi}$, to detect sparse anomalous objects. The function $\psi$ is a sparsity-promoting function and can be chosen as a weakly convex function. 
    
    \item  We propose a PnP version of the proximal block coordinate descent algorithm, called the PnP-PBCD method, for solving the proposed nonconvex nonsmooth minimization problem with an orthogonal constraint.  The subproblems have either closed-form solutions or are easy to compute. We prove that any accumulation point of the sequence generated by the proposed algorithm is a stationary point. 
    
    \item We demonstrate that the proposed PnP-PBCD method outperforms other state-of-the-art methods in detecting anomalous objects in HSIs, even in the presence of noise. 
\end{itemize}

The rest of this paper is organized as follows. In section~\ref{sec:model}, we propose an optimization model for hyperspectral anomaly detection. To solve the proposed model, we propose a PnP-PBCD method in section~\ref{sec:algorithm} and conduct its convergence analysis in section~\ref{sec:convergence}. Next, we conduct experiments in section~\ref{sec:experiments}. The concluding remarks are given in section~\ref{sec:conclusions}.

\section{Optimization Model for Hyperspectral Anomaly Detection}\label{sec:model}

In this section, we propose an optimization model with an implicit deep prior for anomaly detection in noisy HSIs. We first introduce the notations that we use in this paper.

\subsection{Notations}
For a third order tensor $\mathcal{X} \in \mathbb{R}^{n_{1}\times n_{2}\times n_{3}}$, we let $x_{i_{1}i_{2}i_{3}}$ denote its $(i_{1},i_{2},i_{3})$-th entry,  let $x_{i_{1}i_{2}:}$ denote its $(i_{1},i_{2})$-th mode-3 fiber and  let $X_{::i_{3}}$ denote its $i_{3}$-th frontal slice. The mode-$k$ unfolding of a third order tensor $\mathcal{X}$ is denoted as
$\mathcal{X}_{(k)} = {\rm unfold}_{(k)}(\mathcal{X})$, which is the process to linearize all indexes except index $k$. The dimensions of $\mathcal{X}_{(k)}$ are $n_{k}\times \prod^{3}_{j=1,j\ne k}n_{j}$. An element $x_{i_{1}i_2i_{3}}$ of $\mathcal{X}$ corresponds to the position of $(i_{k},j)$ in matrix $\mathcal{X}_{(k)}$, where $j=1+\sum^{3}_{l=1,l\ne k}(i_{l}-1)\prod^{l-1}_{m=1, m\ne k}n_{m}$.  The inverse process of the mode-$k$ unfolding of a tensor $\mathcal{X}$ is denoted by $\mathcal{X} = {\rm fold}_{(k)}(\mathcal{X}_{(k)})$. 
In particular, the mode-$3$ product of a tensor $\mathcal{Z}\in \mathbb{R}^{n_{1}\times n_{2}\times r}$ and a matrix $Y\in \mathbb{R}^{n_3 \times r}$, denoted by $\mathcal{Z}\times_3 Y$, is a tensor $\mathcal{X}\in \mathbb{R}^{n_{1}\times n_{2}\times n_{3}}$ with entries 
$$x_{i_1i_2 j}=\sum_{i_3=1}^{r} z_{i_1i_2i_3}y_{ji_3}.$$
The expression $\mathcal{X}=\mathcal{Z}\times_{3} Y$ can also be written in terms of unfolding of tensors, i.e., $\mathcal{X}_{(3)}=Y\mathcal{Z}_{(3)}$.

For $f:\mathbb{R}^d\to (-\infty,+\infty]$ being a proper and lower semicontinuous function with a finite lower bound function. The function $f$ is $\mu$-strongly convex if $f-\frac{\mu}{2}\|\cdot\|^2$ is convex with $\mu\geq 0$; $f$ is $\rho$-weakly convex if $f+\frac{\rho}{2}\|\cdot\|^2$ is convex with $\rho\geq 0$. The proximal operator of $f$ with parameter $\lambda>0$ evaluated at $x\in \mathbb{R}^d$, denoted as $\text{prox}_{\lambda f}(x)$, is defined as
\[
\text{prox}_{\lambda f}(x): = \underset{u\in \mathbb{R}^d}{\arg\!\min} \left[f(u) + \frac{1}{2\lambda} \|u - x\|_2^2\right].
\]
Note that $\text{prox}_{\lambda f}$ is a set-valued map, when the minimizer is not unique. 

\subsection{Formulation for anomaly detection in noisy HSIs}

According to the high spectral correlation of HSIs, a clean HSI can be expressed in a low-rank tensor representation. Specifically, for  $\mathcal{L}\in \mathbb{R}^{n_1\times n_2\times n_3}$, where  $n_1$ and $n_2$ are the spatial dimensions, and $n_3$ is the spectral dimension, $\mathcal{L}$ can be represented as follows
\begin{equation*}
\mathcal{L}=\mathcal{Z}\times_3 E,
\end{equation*}
where $E \in \mathbb{R}^{n_3\times r }$ represents a basis of the spectral subspace, and the tensor $\mathcal{Z} \in \mathbb{R}^{n_1\times n_2\times r}$ denotes the representation coefficient of $\mathcal{L}$ with respect to $E$.  In particular, we choose $E$ as an orthogonal basis, that is, $E^{\top} E = {I}_r$ with ${I}_r$ denoting the identity matrix of size $r$. Additionally, each band of $\mathcal{Z}$, denoted as $Z_{::n}$, is called as an eigenimage, where $n=1,2,\dots,r$. Then a noisy HSI $\mathcal{O}\in \mathbb{R}^{n_{1}\times n_{2}\times n_{3}}$ can be formulated mathematically as 
\begin{displaymath}\label{1.2}
    \mathcal{O}=\mathcal{Z}\times_3 E+\mathcal{S}+\mathcal{N},
\end{displaymath}
where $\mathcal{S}\in \mathbb{R}^{n_1\times n_2\times n_3}$ represents the sparse components such as anomalous objects, and $\mathcal{N}\in \mathbb{R}^{n_1\times n_2\times n_3}$ represents Gaussian noise.

To remove Gaussian noise and detect anomalous objects simultaneously,  we propose an optimization model as follows
\begin{equation}\label{model:HSI2}
	\begin{split}
		\min_{\mathcal{Z},E,\mathcal{S}}\;&\frac{\delta}{2} \|\mathcal{Z}\times_3 E+\mathcal{S}-\mathcal{O}\|_{F}^2+\tau \|\mathcal{S}\|_{2,\psi}+\Phi_{\Sigma}(\mathcal{Z}),\\
		s.t.\; & E^{\top}E={I}_r,
	\end{split}
\end{equation}
where $\|\cdot\|_{2,\psi}$ represents a generalized group sparsity measure for detecting anomalous objects, $\Phi_{\Sigma}(\cdot)$ represents a proximal denoiser prior for removing Gaussian noise, $E$ is a learnable orthogonal basis, and $\delta,\tau>0$ are parameters. The resulting model is a nonconvex nonsmooth minimization problem with an orthogonal constraint. More details of this model will be provided in the following subsections.

\subsection{Generalized group sparsity measure} As the anomalous objects may show spectral information inconsistent with the nearby background, these objects can be viewed as sparse components $\mathcal{S}$. Grouping $\mathcal{S}$ along the spectral direction at spatial position $(i,j)$, we can measure the magnitude of $s_{ij:}$ using the $\ell_2$ norm, and measure the group sparsity using a sparsity-promoting function~\cite{shen2019structured} for $\psi$. The resulting generalized group sparsity measure can be formulated as follows
    \begin{equation} \label{eq:L2pNorm}
\|\mathcal{S}\|_{2,\psi}=\sum_{i=1}^{n_1}\sum_{j=1}^{n_2}\psi\left(\|s_{ij:}\|_{2}\right),
    \end{equation}
where $\psi:\mathbb{R}\to[0,+\infty)$ and $\|s_{ij:}\|_{2}=\left(\sum_{n=1}^{n_3}s^{2}_{ijn}\right)^{\frac{1}{2}}$. The following are some examples of sparsity-promoting functions for $\psi$:
\begin{enumerate}
    \item[(i)] $\ell_1$ norm: $\psi(t)=|t|$;
    \item[(ii)] Relaxed $\ell_p$ norm: $\psi(t)=(|t|+\varepsilon)^p-\varepsilon^p$, $p\in(0,1)$, $\varepsilon>0$;
    \item[(iii)] Minimax concave penalty (MCP)~\cite{10.1214/09-AOS729}: for $\theta >\lambda$, 
    $$\psi_{\lambda,\theta}(t)=\begin{cases}\lambda|t|-\frac{t^2}{2\theta},\quad& |t|\leq \theta \lambda;\\ \frac{\theta \lambda^2}{2}, &\mbox{otherwise};\end{cases}$$
    \item[(iv)] Smoothly clipped absolute deviation (SCAD)~\cite{fan1997comments}: for $\lambda>0$ and $\theta>2$,
    $$\psi_{\lambda,\theta}(t)=\begin{cases}\lambda|t|,\quad& |t|\leq \lambda;\\ 
    \frac{-t^2+2\theta\lambda|t|-\lambda^2}{2(\theta-1)}\quad& \lambda<|t|\leq  \theta\lambda;\\ 
    \frac{(\theta+1) \lambda^2}{2}, &\mbox{otherwise}.\end{cases}$$
\end{enumerate}
Note that (i) is convex, (ii) is $p\varepsilon^{p-1}$-weakly convex, (iii) is $\frac{1}{\theta}$-weakly convex, and (iv) is $\frac{1}{\theta-1}$-weakly convex, according to~\cite{bohm2021variable}.

\subsection{Proximal denoiser prior in a PnP framework}

If the observed HSI $\mathcal{O}$ is degraded by Gaussian noise, the components of its tensor decomposition will also contain some noise, i.e.,
\begin{equation}\label{eq:noisydecomposition}
    \mathcal{O}= \widetilde{\mathcal{Z}}\times_3 \widetilde{E},
\end{equation}
where $\widetilde{\mathcal{Z}}$ denotes the eigenimages degraded by noise, and $\widetilde{E}$ denotes the orthogonal basis with bias. For illustration, in Figure \ref{fig:eigenimages} we present some selected eigenimages of a noisy HSI with a noise level of $0.03$ using the subspace obtained by the HySime algorithm~\cite{bioucas2008hyperspectral}. As shown in Figure \ref{fig:eigenimages},  the first eigenimage is clean, while the noise level of the other eigenimages increases as the band index increases. To remove the noise in the HSI, we utilize a deep denoiser $\mathcal{D}_{\sigma}$ on each eigenimage with an adaptive noise level $\sigma$.
\begin{figure}[h]
	\centering
    \begin{subfigure}[t]{0.242\linewidth}
    \includegraphics[width=\linewidth]{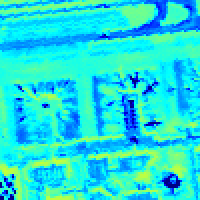}
    \caption{1st eigenimage}
    \end{subfigure}	
    \begin{subfigure}[t]{0.242\linewidth}
    \includegraphics[width=\linewidth]{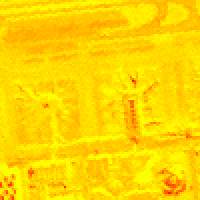}
    \caption{2nd eigenimage}
    \end{subfigure}
    \begin{subfigure}[t]{0.242\linewidth}
    \includegraphics[width=\linewidth]{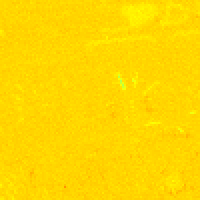}
    \caption{5th eigenimage}
    \end{subfigure}
    \begin{subfigure}[t]{0.242\linewidth}
    \includegraphics[width=\linewidth]{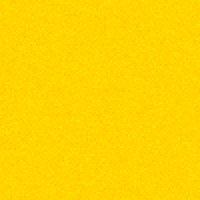}
    \caption{10th eigenimage}
    \end{subfigure}
	\caption{Illustration of eigenimages obtained from a noisy HSI.}\label{fig:eigenimages}
\end{figure}

The deep denoiser $\mathcal{D}_{\sigma}$ that we will use is a proximal denoiser proposed by Hurault et al. in~\cite{hurault2022proximal}, which has the form of a gradient descent step:
\begin{equation}\label{eq:Dsigma}
    \mathcal{D}_{\sigma}:=\text{Id}-\nabla g_{\sigma},
\end{equation}
where $g_{\sigma}$ denotes a smooth parameterized neural network. In particular, $g_{\sigma}$ is defined as follows
\begin{align*}\label{eq:gsigma}
    g_{\sigma}(X)=\frac{1}{2}\|X-\mathcal{N}_{\sigma}(X)\|^2,
\end{align*}
with $\mathcal{N}_{\sigma}(X)$ being a $\mathcal{C}^2$ neural network, specifically DRUNet~\cite{zhang2021plug}, pre-trained for denoising grayscale or color images. Moreover, the denoiser is carefully trained to ensure that $g_{\sigma}$ approximately has an $L_{g_{\sigma}}$-Lipschitz gradient with $L_{g_{\sigma}}<1$. The overall denoiser $\mathcal{D}_{\sigma}$ is called a proximal denoiser, because it behaves like a proximal operator as shown in Proposition~\ref{thm:prox}. More discussions on this proximal denoiser can also be found in~\cite{huang2024deep,Tan2024Provably,wu2024extrapolated}.
\begin{proposition}{(See \cite[Prop. 1]{Hurault2023Relaxed})}\label{thm:prox}
    Let $g_{\sigma}:\mathbb{R}^{n_1\times n_2}\to \mathbb{R}$ be a $\mathcal{C}^2$ function with $\nabla g_{\sigma}$ being $L_{g_{\sigma}}$-Lipschitz and $L_{g_{\sigma}}<1$. Then, for $\mathcal{D}_{\sigma}$ defined as in~\eqref{eq:Dsigma}, there exists a potential $\phi_{\sigma}:\mathbb{R}^{n_1\times n_2}\to [0,+\infty)$ such that $\operatorname{prox}_{\phi_{\sigma}}$ is one-to-one and
    \begin{equation*}
        \mathcal{D}_{\sigma}=\operatorname{prox}_{\phi_{\sigma}},
    \end{equation*}
    where 
    \begin{equation*}\label{eq:phisigma}
        \phi_{\sigma}(X)=\begin{cases}
            g_{\sigma}(\mathcal{D}^{-1}_{\sigma}(X))-\frac{1}{2}\|\mathcal{D}^{-1}_{\sigma}(X)-X\|^2,\quad & \mbox{if } X\in\operatorname{Im}(\mathcal{D}_{\sigma})\\
           +\infty,\quad & \mbox{otherwise}.
        \end{cases}
    \end{equation*}
    Moreover, $\phi_{\sigma}$ is $\frac{L_{g_{\sigma}}}{L_{g_{\sigma}}+1}$-weakly convex and $\phi_{\sigma}(X)\geq g_{\sigma}(X)$ for $\forall X\in \mathbb{R}^{n_1\times n_2}$.
\end{proposition}

To better adapt the proximal denoiser $\mathcal{D}_{\sigma}$ to denoise each eigenimage $Z_{::n}$, we first consider a relaxed version of the proximal denoiser discussed in~\cite{Hurault2023Relaxed} as follows
\begin{align*}
    \mathcal{D}_{\sigma}^{\gamma} = \gamma\mathcal{D}_{\sigma} + (1-\gamma) \operatorname{Id} = \operatorname{Id} -\gamma\nabla g_{\sigma},
\end{align*}
with parameter $\gamma\in[0,1]$. Then by applying Proposition~\ref{thm:prox} with $g^{\gamma}_{\sigma} = \gamma g_{\sigma}$, we get that there exists a $\frac{\gamma L_{g_{\sigma}}}{\gamma L_{g_{\sigma}}+1}$-weakly convex function $\phi_{\sigma}^{\gamma}$ such that $\mathcal{D}_{\sigma}^{\gamma} = \operatorname{prox}_{\phi_{\sigma}^{\gamma}}$ if $\gamma L_{g_{\sigma}} < 1$.  This allows us to control the weak convexity of the regularization function $\phi_{\sigma}^{\gamma}$, leading to a wide range for the selection of step size in the algorithm that we will propose in the next section.

Second, as eigenimages may not fall in $[0,1]$, we consider a shifted denoiser for an eigenimage $Z$ as follows
\begin{align}\label{eq:denoiser}
    \widetilde{\mathcal{D}}_{\sigma}^{\gamma}(Z) = \frac{1}{a}\left[\mathcal{D}_{\sigma}^{\gamma}(aZ+b)-b\right],
\end{align}
where $a>0$ and $b$ is a constant. Then
\begin{align*}
    \widetilde{\mathcal{D}}_{\sigma}^{\gamma}(Z) =\frac{1}{a}\left[ \operatorname{prox}_{\phi_{\sigma}^{\gamma}} (aZ+b)-b\right]=\operatorname{prox}_{\widetilde{\phi}_{\sigma}^{\gamma}}(Z),
\end{align*}
where $\widetilde{\phi}_{\sigma}^{\gamma}(Z)=\frac{1}{a^2}\phi_{\sigma}^{\gamma}(aZ+b)$ and $\widetilde{\phi}_{\sigma}^{\gamma}$ is $\frac{\gamma L_{g_{\sigma}}}{\gamma L_{g_{\sigma}}+1}$-weakly convex.

Lastly, we define the proximal denoiser prior $\Phi_{\Sigma}(\mathcal{Z})$ used in our proposed model~\eqref{model:HSI2} for removing noise in eigenimages as follows
\begin{equation}\label{eq:PhiSigma}
    \Phi_{\Sigma}(\mathcal{Z})=\lambda\sum_{n=1}^r\widetilde{\phi}_{\sigma_n}^{\gamma}({Z}_{::n}),
\end{equation}
where $\lambda>0$ and $\Sigma:=\operatorname{diag}(\sigma_1,\sigma_2,\dots,\sigma_{r})$ with $\sigma_n$ denoting the noise level of the $n$-th eigenimage ${Z}_{::n}$.

\subsection{Learnable orthogonal basis} The orthogonal constrained set on $E$ is also called the Stiefel manifold, defined as $\mathbb{S}_{n_3,r}:=\{E\in\mathbb{R}^{n_3\times r}:E^{\top}E=I_r\}$ with $ n_3\geq r$. By choosing an orthogonal basis $E$, the eigenimages ${Z}_{::n}$ are linearly independent to each other. This allows us to apply the denoiser to each eigenimage ${Z}_{::n}$ independently and the noise covariance matrix is a diagonal matrix.

In some existing works, some subspace methods consider a fixed basis for the low-rank tensor decomposition. However, 
according to~\eqref{eq:noisydecomposition}, this may result in false labels. Hence, we consider learnable basis $E$, which will be updated iteratively.


\section{Plug-and-play Proximal Block Coordinate Descent Method}\label{sec:algorithm}

In this section, we propose a PnP-PBCD method for solving model \eqref{model:HSI2}, which is a nonconvex and nonsmooth optimization problem over a Stiefel manifold. In particular, in model \eqref{model:HSI2}, both $\|\cdot\|_{2,\psi}$ and $\Phi_{\Sigma}$ are weakly convex functions. 

Let $F(\mathcal{Z},E,\mathcal{S})$ denote the objective function of the proposed model~\eqref{model:HSI2} as follows
	\begin{equation}\label{eq:DefF}
F(\mathcal{Z},E,\mathcal{S})=H(\mathcal{Z},E,\mathcal{S})+\tau \|\mathcal{S}\|_{2,\psi}+\Phi_{\Sigma}(\mathcal{Z}),
	\end{equation}
where
\begin{equation*}
    H(\mathcal{Z},E,\mathcal{S})=\frac{\delta}{2} \|\mathcal{Z}\times_3 E+\mathcal{S}-\mathcal{O}\|_{F}^2.
\end{equation*}
Then a PnP-PBCD algorithm for problem \eqref{model:HSI2} is summarized as follows:
\begin{align}
    \mathcal{S}^{k+1}&\in  \underset{\mathcal{S}}{\arg\!\min}\; H(\mathcal{Z}^k,E^k,\mathcal{S})+\tau \|\mathcal{S}\|_{2,\psi}+\frac{\alpha_{\mathcal{S}}}{2}\|\mathcal{S}-\mathcal{S}^k\|_F^2\label{eq:PBCD-S}\\
     E^{k+1}&\in\underset{E \in \mathbb{S}_{n_3,r}}{\arg\!\min}\; H(\mathcal{Z}^k,E,\mathcal{S}^{k+1})+\frac{\alpha_E}{2}\|E-E^k\|_F^2\label{eq:PBCD-E}\\
    \mathcal{Z}^{k+1}&=\underset{\mathcal{Z}}{\arg\!\min}\; H(\mathcal{Z},E^{k+1},\mathcal{S}^{k+1})+\Phi_{\Sigma}(\mathcal{Z})+\frac{\alpha_{\mathcal{Z}}}{2}\|\mathcal{Z}-\mathcal{Z}^k\|_F^2,\label{eq:PBCD-Z}
\end{align}
where the step sizes $\alpha_{\mathcal{S}},\alpha_{\mathcal{Z}}\geq 0$ and $\alpha_E>0$.

In the following, we present the details for computing each update. We will conduct a convergence analysis for the proposed PnP-PBCD algorithm in the next section.

 \subsection{The update of $\mathcal{S}$} Combining the function $H$ and the proximal term in~\eqref{eq:PBCD-S}, the update $\mathcal{S}^{k+1}$ can be written in terms of the proximal operator of $\|\cdot\|_{2,\psi}$ as follows
\begin{align}\label{eq:Supdate}
    \mathcal{S}^{k+1}\in  \operatorname{prox}_{\tilde{\tau}\|\cdot\|_{2,\psi}}\left[ \mathcal{S}^k-\tilde{\alpha}_{\mathcal{S}}\left(\mathcal{S}^k+\mathcal{Z}^k\times_3 E^k- \mathcal{O}\right)\right],
    \end{align}
 where $\tilde{\tau}=\frac{\tau}{\delta+\alpha_{\mathcal{S}}}$ and $\tilde{\alpha}_{\mathcal{S}}=\frac{\delta}{\delta+\alpha_{\mathcal{S}}}$. 
Since $\|\cdot\|_{2,\psi}$ is separable, the $(i,j)$-th mode-$3$ fiber of $\mathcal{S}^{k+1}$ can be computed via
\begin{equation*}
    s^{k+1}_{ij:}\in \operatorname{prox}_{\tilde{\tau}\psi\circ\|\cdot\|_2} (\hat{s}^k_{ij:}),
\end{equation*}
where $\hat{s}^k_{ij:}$ is the $(i,j)$-th mode-$3$ fiber of $\widehat{\mathcal{S}}^k=\mathcal{S}^k-\tilde{\alpha}_{\mathcal{S}}\left(\mathcal{S}^k+\mathcal{Z}^k\times_3 E^k- \mathcal{O}\right)$.
According to~\cite{liu2024orthogonal} and Theorem 4.1 in~\cite{yu2024generalized}, we have
\begin{align*}
\operatorname{prox}_{\tilde{\tau}\psi\circ\|\cdot\|_2} (s)=\begin{cases}
    \operatorname{prox}_{\tilde{\tau}\psi}(\|s\|_2)\frac{s}{\|s\|_2},\quad &\|s\|_2\neq 0\\
    0,\quad &\|s\|_2= 0.
\end{cases}
\end{align*}
Depending on the choice of $\psi$, the proximal operator $\operatorname{prox}_{\tilde{\tau}\psi\circ\|\cdot\|_2}$ is computed differently.

\subsection{The update of $E$} Before we compute the update of $E$ given in~\eqref{eq:PBCD-E}, we first introduce a lemma for the optimization problems over a Stiefel manifold as follows.
\begin{lemma}\label{thm:equiproj}
    Let $X\in\mathbb{R}^{n\times r}$, $A\in\mathbb{R}^{r\times m}$ and $B\in\mathbb{R}^{n\times m}$. Then the solutions of the following problems are the same:
   \begin{subequations}
   \begin{equation}\label{eq:H1}
            \min_{X\in \mathbb{S}_{n,r}} \frac{1}{2}\|XA-B\|_F^2,
        \end{equation}
    \begin{equation}\label{eq:H2}
            \min_{X\in \mathbb{S}_{n,r}} -\langle X,BA^{\top}\rangle,
        \end{equation}
        and
    \begin{equation}\label{eq:H3}
            \min_{X\in \mathbb{S}_{n,r}} \frac{1}{2}\|X-BA^{\top}\|_F^2.
        \end{equation}
        \end{subequations}

\end{lemma}

\begin{proof} If $X\in \mathbb{S}_{n,r}$, i.e., $X^{\top}X=I_r$, we have
        \begin{align}
        \|XA-B\|_F^2=&\|A\|_F^2-2\langle X,BA^{\top}\rangle+\|B\|_F^2\nonumber\\
        =&\|X-BA^{\top}\|_F^2-\|BA^{\top}\|_F^2-r^2+\|A\|_F^2+\|B\|_F^2.\label{eq:XAB}
    \end{align}
    Then the equivalence is immediately achieved.
\end{proof}

Recall that for unfolding of tensors 
\begin{equation*}
	\mathcal{L}=\mathcal{Z}\times_{3} E \quad \mbox{ if and only if }\quad \mathcal{L}_{(3)}=E\mathcal{Z}_{(3)}.
\end{equation*}
Then the function $H$ can be rewritten as $H(\mathcal{Z},E,\mathcal{S})=\frac{\delta}{2} \|E\mathcal{Z}_{(3)} +(\mathcal{S}-\mathcal{O})_{(3)}\|_{F}^2$, which has the same form as~\eqref{eq:H1}. It follows from Lemma~\ref{thm:equiproj} that minimizing \eqref{eq:H1} is equivalent to minimizing more simple forms as~\eqref{eq:H2} and \eqref{eq:H3}. In particular, we consider $\widetilde{H}(\mathcal{Z},E,\mathcal{S})=-\delta\langle E,(\mathcal{O}-\mathcal{S})_{(3)}(\mathcal{Z}_{(3)})^{\top}\rangle$ of the form as~\eqref{eq:H2}. Then the update of $E$  can be computed as follows
\begin{align}
E^{k+1}&\in\underset{E \in \mathbb{S}_{n_3,r}}{\arg\!\min}\; \widetilde{H}(\mathcal{Z}^k,E,\mathcal{S}^{k+1})+\frac{\alpha_E}{2}\|E-E^k\|_F^2,\label{Eq:Htilde}\\
&=\operatorname{Proj}_{\mathbb{S}_{n_3,r}}\left[E^k+\tilde{\alpha}_{E}(\mathcal{S}^{k+1}-\mathcal{O})_{(3)}(\mathcal{Z}_{(3)}^k)^{\top}\right],\nonumber
\end{align}
where $\tilde{\alpha}_{E}=\frac{\delta}{\alpha_E}$ and $\operatorname{Proj}_{\mathbb{S}_{n_3,r}}$ denotes the projection onto the Stiefel manifold $\mathbb{S}_{n_3,r}$. 
It follows from Lemma 3.1 in \cite{liu2024orthogonal} that the projection $\operatorname{Proj}_{\mathbb{S}_{n_3,r}}$ has a closed form and $E^{k+1}$ can be computed as follows 
\begin{align}\label{eq:Eupdate}
E^{k+1}=U^{k+1}(V^{k+1})^{\top}, \quad \mbox{with }U^{k+1}\widehat{\Sigma}^{k+1} (V^{k+1})^{\top}=\widehat{E}^k
\end{align}
where $\widehat{E}^k=E^k+\tilde{\alpha}_{E}(\mathcal{S}^{k+1}-\mathcal{O})_{(3)}(\mathcal{Z}_{(3)}^k)^{\top}$, $U^{k+1}\widehat{\Sigma}^{k+1} (V^{k+1})^{\top}$ is a reduced SVD of $\widehat{E}^k$, $U^{k+1}\in \mathbb{R}^{n_3\times r}$, $V^{k+1}\in \mathbb{R}^{r\times r}$, and $\widehat{\Sigma}^{k+1}\in \mathbb{R}^{r\times r}$.

\subsection{The update of $\mathcal{Z}$}

By applying $E^{\top}E=I_r$, we have
\begin{equation*}
    \|\mathcal{Z}\times_{3} E-\mathcal{L}\|_F^2=\|\mathcal{Z}-\mathcal{L}\times_{3} E^{\top}\|_F^2.
\end{equation*}
Then the subproblem for updating $\mathcal{Z}$ in \eqref{eq:PBCD-Z} can be reformulated as
\begin{align}
   \mathcal{Z}^{k+1}&\in  \underset{\mathcal{Z}}{\arg\!\min}\;\Phi_{\Sigma}(\mathcal{Z})+\frac{\delta}{2}\|\mathcal{Z}+(\mathcal{S}^{k+1}-\mathcal{O})\times_3 (E^{k+1})^{\top}\|_F^2+\frac{\alpha_{\mathcal{Z}}}{2}\|\mathcal{Z}-\mathcal{Z}^{k}\|_F^2,\nonumber\\
    &=\operatorname{prox}_{\tilde{\alpha}_{\mathcal{Z}}\Phi_{\Sigma}}\left[\mathcal{Z}^{k} -\tilde{\alpha}_{\mathcal{Z}}(\mathcal{Z}^{k} -(\mathcal{O}-\mathcal{S}^{k+1})\times_3 (E^{k+1})^{\top}\right],\label{eq:Zupdate}
    \end{align}
where $\tilde{\alpha}_{\mathcal{Z}}=\frac{\delta}{\delta+\alpha_{\mathcal{Z}}}$. By choosing the parameter $\lambda=\frac{1}{\tilde{\alpha}_{\mathcal{Z}}}$ in $\Phi_{\Sigma}$ defined in~\eqref{eq:PhiSigma}, we have $\tilde{\alpha}_{\mathcal{Z}}\Phi_{\Sigma}(\mathcal{Z})=\sum_{n=1}^r\widetilde{\phi}_{\sigma_n}^{\gamma}({Z}_{::n})$. Then each eigenimage, i.e., ${Z}^{k+1}_{::n}$ can be computed via the shifted and relaxed proximal denoiser as follows
\begin{equation*}
    {Z}^{k+1}_{::n}=\widetilde{\mathcal{D}}_{\sigma_n}^{\gamma}(\widehat{Z}^{k}_{::n}),
\end{equation*}
where $\widehat{\mathcal{Z}}^{k}=\mathcal{Z}^{k} -\tilde{\alpha}_{\mathcal{Z}}(\mathcal{Z}^{k} -(\mathcal{O}-\mathcal{S}^{k+1})\times_3 (E^{k+1})^{\top})$.

All in all, the proposed PnP-PBCD algorithm for model~\eqref{model:HSI2} is summarized in Algorithm~\ref{alg1}.
\begin{algorithm}[H]   
	\renewcommand{\algorithmicrequire}{\textbf{Input:}}
	\renewcommand{\algorithmicensure}{\textbf{Output:}}
	\caption{PnP-PBCD algorithm for model~\eqref{model:HSI2}}
	\label{alg1}
	\begin{algorithmic}[1]
		\STATE Initialize $(\mathcal{Z}^0,E^0,\mathcal{S}^0)$ with \((E^0)^{\top}E^0=I_r\);
		\STATE  Set parameters $\alpha_{\mathcal{S}},\alpha_{\mathcal{Z}}\geq 0$ and $\alpha_E>0$;
  \STATE Set $k=0$.
		\REPEAT
		\STATE Compute $\mathcal{S}^{k+1}$ by \eqref{eq:Supdate};
  \STATE Compute $E^{k+1}$ by \eqref{eq:Eupdate};
  \STATE Compute $\mathcal{Z}^{k+1}$ by \eqref{eq:Zupdate};
		\STATE $k\leftarrow k+1$.
		\UNTIL the stopping criterion is met.
		\ENSURE $(\mathcal{Z}^k,E^k,\mathcal{S}^k)$. \end{algorithmic}  
\end{algorithm}


\section{Convergence Analysis of the PnP-PBCD Method}\label{sec:convergence}

In this section, we conduct a convergence analysis on the proposed PnP-PBCD method. We first define the first-order optimality condition of problem~\eqref{model:HSI2} based on the subdifferentials for nonconvex nonsmooth functions and the Riemannian gradient of a smooth function on the Stiefel manifold. Then we prove that any accumulation point of the sequence generated by the PnP-PBCD method given in Algorithm~\ref{alg1} is a stationary point of problem~\eqref{model:HSI2}.

\subsection{First-order optimality condition}

First, we provide some preliminaries on subdifferentials and Riemannian gradients. 

Let $f:\mathbb{R}^d\to(-\infty,+\infty]$ be a proper and lower semicontinuous function with a finite lower bound. The (limiting) subdifferential of $f$ at $x\in \operatorname{dom}f:=\{x\in\mathbb{R}^{d}:f(x)<\infty\}$, denoted by $\partial f(x)$, is defined as
\begin{displaymath}
        \partial f(x)\!:=\{u \in\mathbb{R}^{d}\!: \!\exists x^{k}\to x, f(x^k)\to f(x)  \text{ and } u^{k}\to u \text{ with } u^{k}\in \hat{\partial}f(x^{k}) \mbox{ as }k\to \infty\},
\end{displaymath}
where $\hat{\partial}f(x)$ denotes the Fr\'echet subdifferential of $f$ at $x \in \operatorname{dom}f$, which is the set
of all $u \in \mathbb{R}^{d}$ satisfying
\begin{equation}
    \begin{split}
      \underset{y\ne x,y\rightarrow x}{\lim\inf}\frac{f(y)-f(x)-\langle u,y-x\rangle}{\|y-x\|}\ge 0.
    \end{split}
\end{equation}
One can also observe that $\{u \in\mathbb{R}^{d}: \exists x^{k}\to x, f(x^k)\to f(x)  \text{ and } u^{k}\to u \text{ with } u^{k}\in \partial f(x^{k}) \mbox{ as }k\to \infty\} \subseteq \partial f(x)$.

Also, we set $\mathcal{T}_X \mathbb{S}_{m,n}:=\{Y\in\mathbb{R}^{m\times n}: Y^{\top}X+X^{\top}Y=0\}$
as the tangent space of Stiefel manifold at $X\in\mathbb{R}^{m\times n}$. We also set the Riemannian metric on Stiefel manifold as the metric induced from the Euclidean inner product. Then according to \cite{2008Optimization}, the Riemannian gradient of  a smooth function $f$ at $X$ is given by
\begin{displaymath}
    \operatorname{grad} f(X) :=\operatorname{Proj}_{\mathcal{T}_X \mathbb{S}_{m,n}}(\nabla f(X)),
\end{displaymath}
where $\operatorname{Proj}_{\mathcal{T}_X \mathbb{S}_{m,n}}(Y):= Y  -  \frac{1}{2} X(X^{\top} Y  +  Y^{\top} X)$.

Second, we define the first-order optimality condition of the orthogonal constrained optimization problem as in \eqref{model:HSI2}. The point  \((\bar{\mathcal{Z}},\bar{E},\bar{\mathcal{S}})\) is a first-order stationary point of problem \eqref{model:HSI2} if 
\begin{subequations}
\begin{align}
&{0}\in \nabla_{\mathcal{S}}H(\bar{\mathcal{Z}},\bar{E},\bar{\mathcal{S}}) + \tau\partial\|\cdot\|_{2,\psi}(\bar{\mathcal{S}}),\label{eq:gradXH1}\\
&{0}=\operatorname{grad}_{E} H(\bar{\mathcal{Z}},\bar{E},\bar{\mathcal{S}}),\quad \bar{E}^{\top}\bar{E}={I}_{r},\label{eq:gradXH2}\\
&{0}\in \nabla_{\mathcal{Z}}H(\bar{\mathcal{Z}},\bar{E},\bar{\mathcal{S}})+\partial \Phi_{\Sigma} (\bar{\mathcal{Z}}),\label{eq:gradXH3}
\end{align}
\end{subequations}
where  $\operatorname{grad}_{E} H(\bar{\mathcal{Z}},\bar{E},\bar{\mathcal{S}})$ denotes the Riemannian gradient of $H$ with respect to $E$ evaluated at $(\bar{\mathcal{Z}},\bar{E},\bar{\mathcal{S}})$, and $\partial\|\cdot\|_{2,\psi}$ and $\partial \Phi_{\Sigma}$  denote the subdifferentials of $\|\cdot\|_{2,\psi}$ and $\Phi_{\Sigma}$, respectively.

\subsection{Subsequence convergence} We present some assumptions for problem~\eqref{model:HSI2} as follows:
\begin{itemize}
    \item[(A1)] $\tau\|\cdot\|_{2,\psi}$ is $\rho_1$-weakly convex. 

    \item[(A2)] $\Phi_{\Sigma}$ is coercive and $\rho_2$-weakly convex.

\end{itemize}
Note that the coercivity on $\Phi_{\Sigma}$ required in Assumption~(A2) can be achieved by the coercivity of $g_{\sigma}$, according to~\cite{Hurault2023Relaxed}.

We first show two lemmas that we will use in the convergence analysis. The first lemma is for analyzing the updates of $\mathcal{Z}$ and $\mathcal{S}$, and the second lemma is for understanding the update of $E$.

\begin{lemma}\label{thm:weakstrong}
Let $f:\mathbb{R}^d\to(-\infty,+\infty]$ be a proper, lower semicontinuous  and $\rho$-weakly convex function, and let  $h:\mathbb{R}^d\to \mathbb{R}$ be a differentiable and $\mu$-strongly convex function. If there exists $\hat{x}$ such that
    \begin{equation}\label{eq:updatehatx}
        \hat{x}{ \,\in \,}  \underset{x}{\arg\!\min}\;f(x)+h(x)+\frac{\alpha}{2}\|x-x_0\|^2,
    \end{equation}
   where $\alpha\geq 0$, then
    \begin{equation}\label{eq:Fx0}
    f(x_0)+h(x_0)\geq f(\hat{x})+h(\hat{x}) +\frac{\alpha+(\alpha+\mu-\rho)_+}{2}\|x_0-\hat{x}\|^2,
\end{equation}
with $(\alpha+\mu-\rho)_+=\max\{\alpha+\mu-\rho,0\}$.
\end{lemma}
\begin{proof} 

For the $\rho$-weakly convex function $f$, it follows from Lemma 2.1 in \cite{davis2019stochastic} that
$$
f(x_0)\geq f(\hat{x})+\langle u,x_0-\hat{x}\rangle- \frac{\rho}{2}\|x_0-\hat{x}\|^2,
$$
 for $\forall u\in \partial f(\hat{x})$.  Similarly, for the differentiable and $\mu$-strongly convex function $h$, it follows from \cite{Pol66} that
  $$
h(x_0)\geq h(\hat{x})+ \langle \nabla h(\hat{x}),x_0-\hat{x}\rangle + \frac{\mu}{2}\|x_0-\hat{x}\|^2.
$$  
Summing the inequalities above, we obtain
\begin{equation*}
     f(x_0)+h(x_0)\geq f(\hat{x})+h(\hat{x})+\langle u+\nabla h(\hat{x}),x_0-\hat{x}\rangle + \frac{\mu-\rho}{2}\|x_0-\hat{x}\|^2.
\end{equation*}
By the update of $\hat{x}$ in~\eqref{eq:updatehatx}, we have $0\in \partial f(\hat{x})+\nabla h(\hat{x})+\alpha(\hat{x}-x_0)$. That is, $-\nabla h(\hat{x})+\alpha(x_0-\hat{x})\in \partial f(\hat{x})$. Substituting $u=-\nabla h(\hat{x})+\alpha(x_0-\hat{x})$ into the above inequality, we obtain 
    \begin{equation*}
	f(x_0)+h(x_0)\geq f(\hat{x})+h(\hat{x}) +\frac{2\alpha+\mu-\rho}{2}\|x_0-\hat{x}\|^2.
\end{equation*}
Note that by \eqref{eq:updatehatx}, we also have
    \begin{equation*}
	f(x_0)+h(x_0)\geq f(\hat{x})+h(\hat{x}) +\frac{\alpha}{2}\|x_0-\hat{x}\|^2.
\end{equation*}
Hence, taking the maximum of $2\alpha+\mu-\rho$ and $\alpha$, we obtain \eqref{eq:Fx0}.
\end{proof}

Note that for $\alpha\ge 0$, it can be verified that $\alpha+(\alpha+\mu-\rho)_+>0$ if and only if $\alpha>0$ or $\mu-\rho>0$.

\begin{proposition}\label{thm:gradH} Let $X\in\mathbb{R}^{n\times r}$, $A\in\mathbb{R}^{r\times m}$ and $B\in\mathbb{R}^{n\times m}$. Meanwhile, let $H_1(X)$, $H_2(X)$, and $H_3(X)$ denote the objective functions of~\eqref{eq:H1}, \eqref{eq:H2}, and \eqref{eq:H3}, respectively. Then we have for any $X\in\mathbb{S}_{n,r}$,
\begin{equation}\label{eq:grad=}
    \operatorname{grad} H_1(X)=\operatorname{grad} H_2(X)=\operatorname{grad} H_3(X).
\end{equation}

\end{proposition}

\begin{proof}
    We first compute the gradients of $H_1$, $H_2$, and $H_3$ as follows
    \begin{align*}
        \nabla H_1 (X)&= (XA-B)A^{\top}=XAA^{\top}+\nabla H_2\\
        \nabla H_2 (X)&= -BA^{\top}\\
        \nabla H_3 (X)&= X-BA^{\top}=X+\nabla H_2.\
    \end{align*}
Then by projecting the gradients above onto the tangent space of Stiefel manifold at $X$, we can compute the Riemannian gradients of $H_1$, $H_2$, and $H_3$. Since $\operatorname{Proj}_{\mathcal{T}_{X}\mathbb{S}_{n,r}}$ is linear, we have
\begin{align*}
    \operatorname{grad} H_1(X)&=\operatorname{Proj}_{\mathcal{T}_{X}\mathbb{S}_{n,r}}(XAA^{\top})+\operatorname{grad} H_2(X)
\end{align*}
and
\begin{align*}
    \operatorname{grad} H_3(X)&=\operatorname{Proj}_{\mathcal{T}_{X}\mathbb{S}_{n,r}}(X)+\operatorname{grad} H_2(X).
\end{align*}
It is easy to verify that for any $X\in\mathbb{S}_{n,r}$,
$$\operatorname{Proj}_{\mathcal{T}_{X}\mathbb{S}_{n,r}}(XAA^{\top})=XAA^{\top}- \frac{1}{2} X(X^{\top} XAA^{\top}  +  (XAA^{\top})^{\top} X)=0$$ and $\operatorname{Proj}_{\mathcal{T}_{X}\mathbb{S}_{n,r}}(X)=X- \frac{1}{2} X(X^{\top} X  +  X^{\top} X)=0$. Therefore, we obtain~\eqref{eq:grad=}.
\end{proof}

Next, we prove the non-increasing monotonicity of the objective sequence\\ $\{F(\mathcal{Z}^k,E^k,\mathcal{S}^k)\}$ and the boundedness of the sequence $\{(\mathcal{Z}^k,E^k,\mathcal{S}^k)\}$ generated by Algorithm~\ref{alg1}.

\begin{theorem}\label{Thm:IterateSeq} Assume that assumptions (A1) and (A2) are satisfied. Let $F(\mathcal{Z},E,\mathcal{S})$ be the objective function of model~\eqref{model:HSI2} defined in \eqref{eq:DefF} and
let $\{(\mathcal{Z}^k,E^k,\mathcal{S}^k)\}$ be the sequence generated by Algorithm~\ref{alg1} with $\alpha_{\mathcal{S}}+(\alpha_{\mathcal{S}}+\delta-\rho_1)_+>0$ and $\alpha_{\mathcal{Z}}+(\alpha_{\mathcal{Z}}+\delta-\rho_2)_+>0$.
Then the following statements hold: 
\begin{enumerate}
    \item[(i)]  The sequence $\{F(\mathcal{Z}^k,E^k,\mathcal{S}^k)\}$ of function values at the iteration points decreases monotonically, and
\begin{equation}\label{Eq:Decrease}
    \begin{split}
&F(\mathcal{Z}^{k-1},E^{k-1},\mathcal{S}^{k-1})-F(\mathcal{Z}^{k},E^{k},\mathcal{S}^{k})\\
\ge &\frac{c_1}{2}\left(\|\mathcal{S}^{k}-\mathcal{S}^{k-1}\|^{2}_{F}+\|E^{k}-E^{k-1}\|^{2}_{F}+\|\mathcal{Z}^{k}-\mathcal{Z}^{k-1}\|^{2}_{F}\right),
    \end{split}
\end{equation}
for some $c_1>0$.
\item[(ii)] The sequence $\{(\mathcal{Z}^k,E^k,\mathcal{S}^k)\}$ is bounded.

\item[(iii)] $\displaystyle\lim_{k\to\infty} \|\mathcal{S}^{k}-\mathcal{S}^{k-1}\|_{F}=0$, $\displaystyle\lim_{k\to\infty} \|E^k-E^{k-1} \|_{F}=0$, and $\displaystyle\lim_{k\to\infty}\|\mathcal{Z}^{k}-\mathcal{Z}^{k-1}\|_{F}=0$, for any $i=1,2,3$.
\end{enumerate}
\end{theorem}

\begin{proof} (i) The updates of $\mathcal{S}$ and $\mathcal{Z}$ have the form of~\eqref{eq:updatehatx}. Then Lemma~\ref{thm:weakstrong} implies that
\begin{align*}
    &F(\mathcal{Z}^{k-1},E^{k-1},\mathcal{S}^{k-1})-F(\mathcal{Z}^{k-1},E^{k-1},\mathcal{S}^{k})
    \ge\frac{\alpha_{\mathcal{S}}+(\alpha_{\mathcal{S}}+\delta-\rho_1)_+}{2}\|\mathcal{S}^{k}-\mathcal{S}^{k-1}\|^{2}_{F}
\end{align*}
and 
\begin{align*}
    &F(\mathcal{Z}^{k-1},E^{k},\mathcal{S}^{k})-F(\mathcal{Z}^{k},E^{k},\mathcal{S}^{k})
    \ge\frac{\alpha_{\mathcal{Z}}+(\alpha_{\mathcal{Z}}+\delta-\rho_2)_+}{2}\|\mathcal{Z}^{k}-\mathcal{Z}^{k-1}\|^{2}_{F}.
\end{align*}

Next, it follows from \eqref{eq:XAB} and \eqref{Eq:Htilde}, we have
\begin{align*}
    &F(\mathcal{Z}^{k-1},E^{k-1},\mathcal{S}^{k})-F(\mathcal{Z}^{k-1},E^{k},\mathcal{S}^{k})\\
    =&\widetilde{H}(\mathcal{Z}^{k-1},E^{k-1},\mathcal{S}^{k})-\widetilde{H}(\mathcal{Z}^{k-1},E^{k},\mathcal{S}^{k})\\
    \geq &\frac{\alpha_{E}}{2}\|E^k-E^{k-1}\|_F^2.
\end{align*}
Combining the inequalities above, we  obtain \eqref{Eq:Decrease} with $c_1 =\min\{\alpha_{\mathcal{S}}+(\alpha_{\mathcal{S}}+\delta-\rho_1)_+,\alpha_{E},\alpha_{\mathcal{Z}}+(\alpha_{\mathcal{Z}}+\delta-\rho_2)_+\}$.

(ii) Since $(E^k)^{\top}E^k={I}_{r}$, we have that the sequence $\{E^k\}$ is bounded.
By (i), we have $F(\mathcal{Z}^k,E^{k},\mathcal{S}^{k})\le F(\mathcal{Z}^0,E^0,\mathcal{S}^0)$. Also, we observe that $F(\mathcal{Z}^k,E^{k},\mathcal{S}^{k})\ge \tau \|\mathcal{S}^k\|_{2,\psi}+\Phi_{\Sigma}(\mathcal{Z}^k)\geq 0$. Since both $\|\cdot\|_{2,\psi}$ and $\Phi_{\Sigma}$ are coercive, that is,
\begin{displaymath}
    \lim_{\|\mathcal{S}\|_F\to \infty}\|\mathcal{S}\|_{2,\psi}=\infty \quad\mbox{ and }\quad \lim_{ \|\mathcal{Z}\|_F\to \infty} \Phi_{\Sigma}(\mathcal{Z})=\infty,
\end{displaymath}
we have that the sequences $\{\mathcal{S}^k\}$ and $\{\mathcal{Z}^k\}$ are bounded. 

(iii) Let $K$ be an arbitrary integer. Summing \eqref{Eq:Decrease} from $k=1$ to $K-1$, we have
 \begin{align*}
 &\sum^{K}_{k=1}\left(\|\mathcal{S}^{k}-\mathcal{S}^{k-1}\|^{2}_{F}+\|E^k-E^{k-1}\|_F^2+ \|\mathcal{Z}^{k}-\mathcal{Z}^{k-1}\|^{2}_{F}\right)\\
 \le& \frac{2}{c_1}\left(F(\mathcal{Z}^0,E^0,\mathcal{S}^0)-
 F(\mathcal{Z}^K,E^K,\mathcal{S}^K)\right)\\
 \le&\frac{2}{c_1}F(\mathcal{Z}^0,E^0,\mathcal{S}^0).
\end{align*}
 Taking the limits of both sides of the inequality as $K\rightarrow \infty$, we have $$\sum^{\infty}_{k=1}\left(\|\mathcal{S}^{k}-\mathcal{S}^{k-1}\|^{2}_{F}+\|E^k-E^{k-1}\|_F^2+ \|\mathcal{Z}^{k}-\mathcal{Z}^{k-1}\|^{2}_{F}\right)<\infty.$$ Then assertion (iii) immediately holds.
\end{proof}

\begin{corollary}~\label{Cor:L} Assume that assumptions (A1) and (A2) are satisfied.
Let $\{(\mathcal{Z}^k,E^k,\mathcal{S}^k)\}$ be the sequence generated by Algorithm~\ref{alg1} with $\alpha_{\mathcal{S}}+(\alpha_{\mathcal{S}}+\delta-\rho_1)_+>0$ and $\alpha_{\mathcal{Z}}+(\alpha_{\mathcal{Z}}+\delta-\rho_2)_+>0$.  Then $\displaystyle\lim_{k\to\infty}\|\mathcal{Z}^{k}\times_{3}E^k-\mathcal{Z}^{k-1}\times_{3}E^{k-1}\|_{F}=0$.
\end{corollary}
\begin{proof} We have
\begin{align}
    &\|\mathcal{Z}^{k}\times_{3}E^k-\mathcal{Z}^{k-1}\times_{3}E^{k-1}\|_F\nonumber\\
    \leq &\|(\mathcal{Z}^{k}-\mathcal{Z}^{k-1})\times_{3}E^{k}\|_F+ \|\mathcal{Z}^{k-1}\times_{3}(E^k-E^{k-1})\|_F\nonumber\\
    \leq &\sqrt{r}\|\mathcal{Z}^{k}-\mathcal{Z}^{k-1}\|_F+ c_2\|E^k-E^{k-1}\|_F, \label{eq:ZkEk}
\end{align}
where $c_2=\max_k \|\mathcal{Z}^k\|_F<\infty$ according to assertion (ii) in Theorem~\ref{Thm:IterateSeq}. Then by assertion (iii) in Theorem~\ref{Thm:IterateSeq}, the result holds.

 \end{proof}

Lastly, we show that every convergent subsequence converges to the first-order stationary point of problem \eqref{model:HSI2}.

\begin{theorem}\label{thm:main}
Assume that assumptions (A1) and (A2) are satisfied.
Let $\{(\mathcal{Z}^k,E^k,\mathcal{S}^k)\}$ be the sequence generated by Algorithm~\ref{alg1} with $\alpha_{\mathcal{S}}+(\alpha_{\mathcal{S}}+\delta-\rho_1)_+>0$ and $\alpha_{\mathcal{Z}}+(\alpha_{\mathcal{Z}}+\delta-\rho_2)_+>0$.  Then every accumulation point of $\{(\mathcal{Z}^k,E^k,\mathcal{S}^k)\}$ is a first-order stationary point of problem \eqref{model:HSI2}.
\end{theorem} 
\begin{proof} By the updates of $\mathcal{S}^{k}$, $E^k$ and $\mathcal{Z}^{k}$ given in~\eqref{eq:PBCD-S}, \eqref{eq:PBCD-E}, and \eqref{eq:PBCD-Z}, respectively, we have for any $k=1,2,\dots$
\begin{align*}
&0\in \nabla_{\mathcal{S}} H(\mathcal{Z}^{k-1},E^{k-1},\mathcal{S}^{k}) + \tau\partial \|\cdot\|_{2,\psi}(\mathcal{S}^{k})+\alpha_{\mathcal{S}}(\mathcal{S}^{k}-\mathcal{S}^{k-1}),\\
&0 = \operatorname{grad}_{E} \widetilde{H}(\mathcal{Z}^{k-1},E^{k},\mathcal{S}^{k}),\quad (E^{k})^{\top}E^{k}=I_r,\\
&{0}\in \nabla_{\mathcal{Z}} H(\mathcal{Z}^{k},E^{k},\mathcal{S}^{k})+\partial\Phi_{\Sigma}(\mathcal{Z}^{k})+\alpha_{\mathcal{Z}}(\mathcal{Z}^{k}-\mathcal{Z}^{k-1}).
\end{align*}
Then we have
\begin{subequations} 
\begin{align}
    A_{\mathcal{S}}^k:&=-\alpha_{\mathcal{S}}(\mathcal{S}^{k}-\mathcal{S}^{k-1}) +\nabla_{\mathcal{S}} H(\mathcal{Z}^{k},E^{k},\mathcal{S}^{k})-\nabla_{\mathcal{S}} H(\mathcal{Z}^{k-1},E^{k-1},\mathcal{S}^{k})\label{eq:ASk}\\
    &\in \nabla_{\mathcal{S}} H(\mathcal{Z}^{k},E^{k},\mathcal{S}^{k}) + \tau\partial \|\cdot\|_{2,\psi}(\mathcal{S}^{k})\nonumber\\
    A_E^k:&=\operatorname{grad}_{E} \widetilde{H}(\mathcal{Z}^{k},E^{k},\mathcal{S}^{k})-\operatorname{grad}_{E} \widetilde{H}(\mathcal{Z}^{k-1},E^{k},\mathcal{S}^{k})\label{eq:AEk}\\
    &=\operatorname{grad}_{E} H(\mathcal{Z}^{k},E^{k},\mathcal{S}^{k}),\nonumber
\end{align}
and
\begin{align}
    A_{\mathcal{Z}}^k:&=-\alpha_{\mathcal{Z}}(\mathcal{Z}^{k}-\mathcal{Z}^{k-1})\label{eq:AZk}\\
    &\in\nabla_{\mathcal{Z}} H(\mathcal{Z}^{k},E^{k},\mathcal{S}^{k})+\partial\Phi_{\Sigma}(\mathcal{Z}^{k}).\nonumber
\end{align}
\end{subequations}
Note that \eqref{eq:AEk} is obtained by Proposition~\ref{thm:gradH}. Furthermore, since $\nabla_{\mathcal{S}} H(\mathcal{Z},E,\mathcal{S})=\delta (\mathcal{Z}\times_3 E+\mathcal{S}-\mathcal{O})$, by~\eqref{eq:ZkEk} we have
\begin{align}
    \|A_{\mathcal{S}}^k\|_F&\leq \alpha_{\mathcal{S}}\|\mathcal{S}^{k}-\mathcal{S}^{k-1}\|_F+\delta\|\mathcal{Z}^{k}\times_{3}E^k-\mathcal{Z}^{k-1}\times_{3}E^{k-1}\|_F\nonumber\\
    &\leq \alpha_{\mathcal{S}}\|\mathcal{S}^{k}-\mathcal{S}^{k-1}\|_F+\delta\sqrt{r}\|\mathcal{Z}^{k}-\mathcal{Z}^{k-1}\|_F+ c_2\delta\|E^k-E^{k-1}\|_F.\label{eq:whyE}
\end{align}
Since $\operatorname{grad}_{E} \widetilde{H}(\mathcal{Z},E,\mathcal{S})=\operatorname{Proj}_{\mathcal{T}_{E}\mathbb{S}_{n_3,r}}\left(-\delta(\mathcal{O}-\mathcal{S})_{(3)}(\mathcal{Z}_{(3)})^{\top}\right)$, we have
\begin{align*}
    \|A_E^k\|_F&\leq \left\|\operatorname{Proj}_{\mathcal{T}_{E^k}\mathbb{S}_{n_3,r}}\left(-\delta(\mathcal{O}-\mathcal{S}^k)_{(3)}(\mathcal{Z}^k_{(3)}-\mathcal{Z}^{k-1}_{(3)})^{\top}\right)\right\|_F\\
    &\leq (1+r^2)\|\delta(\mathcal{O}-\mathcal{S}^k)_{(3)}(\mathcal{Z}^k_{(3)}-\mathcal{Z}^{k-1}_{(3)})^{\top}\|_F\\
    &\leq (1+r^2)\delta c_3\|\mathcal{Z}^{k}-\mathcal{Z}^{k-1}\|_F,
\end{align*}
where $c_3=\max_k \|\mathcal{O}-\mathcal{S}^k\|_F$. Also, we have
\begin{align*}
    \|A_{\mathcal{Z}}^k\|_F\leq \alpha_{\mathcal{Z}}\|\mathcal{Z}^{k}-\mathcal{Z}^{k-1}\|_F.
\end{align*}

Suppose that $\{(\mathcal{Z}^{k_l},E^{k_l},\mathcal{S}^{k_l})\}$ is a subsequence of $\{(\mathcal{Z}^k,E^k,\mathcal{S}^k)\}$ which converges to $(\bar{\mathcal{Z}},\bar{E},\bar{\mathcal{S}})$ as $l\to\infty$.  It immediately follows from Theorem~\ref{Thm:IterateSeq} (iii) that $A_{\mathcal{S}}^{k_l}\to 0$, $A_{E}^{k_l}\to 0$, and $A_{\mathcal{Z}}^{k_l}\to 0$, as $l\to\infty$. Also, due to the continuity of $\nabla_{\mathcal{S}} H$, $\operatorname{grad}_{E} H$ and $\nabla_{\mathcal{Z}} H$, we have 
\begin{align*}
    \nabla_{\mathcal{S}} H(\mathcal{Z}^{k_l},E^{k_l},\mathcal{S}^{k_l})&\to \nabla_{\mathcal{S}} H(\bar{\mathcal{Z}},\bar{E},\bar{\mathcal{S}})\\
    \operatorname{grad}_{E} H(\mathcal{Z}^{k_l},E^{k_l},\mathcal{S}^{k_l})&\to \operatorname{grad}_{E} H(\bar{\mathcal{Z}},\bar{E},\bar{\mathcal{S}})\\
    \nabla_{\mathcal{Z}} H(\mathcal{Z}^{k_l},E^{k_l},\mathcal{S}^{k_l})&\to \nabla_{\mathcal{Z}} H(\bar{\mathcal{Z}},\bar{E},\bar{\mathcal{S}}),
\end{align*}
as $l\to\infty$. Then $(\bar{\mathcal{Z}},\bar{E},\bar{\mathcal{S}})$ satisfies \eqref{eq:gradXH2}  for the first-order optimality condition of problem \eqref{model:HSI2}. 

Next, we show $(\bar{\mathcal{Z}},\bar{E},\bar{\mathcal{S}})$ satisfies \eqref{eq:gradXH1} and \eqref{eq:gradXH3}. Since $\|\cdot\|_{2,\psi}$ is lower semicontinuous, we have
\begin{align*}
    \underset{l\to\infty}{\lim\inf}\; \|\mathcal{S}^{k_l}\|_{2,\psi}\geq \|\bar{\mathcal{S}}\|_{2,\psi}.
\end{align*}
Then according to~\eqref{eq:PBCD-S}, we have
\begin{align*}
&H(\mathcal{Z}^{k_l-1},E^{k_l-1},\mathcal{S}^{k_l})+\tau\|\mathcal{S}^{k_l}\|_{2,\psi}+\frac{\alpha_{\mathcal{S}}}{2}\|\mathcal{S}^{k_l}-\mathcal{S}^{k_l-1}\|_F^2\\
\leq &    H(\mathcal{Z}^{k_l-1},E^{k_l-1},\bar{\mathcal{S}})+\tau\|\bar{\mathcal{S}}\|_{2,\psi}+\frac{\alpha_{\mathcal{S}}}{2}\|\bar{\mathcal{S}}-\mathcal{S}^{k_l-1}\|_F^2.
\end{align*}
The continuity of $H$ implies that $\lim_{l\to\infty} H(\mathcal{Z}^{k_l-1},E^{k_l-1},\bar{\mathcal{S}}) -H(\mathcal{Z}^{k_l-1},E^{k_l-1},\bar{\mathcal{S}})=0$. By rewriting the above inequality and taking the limit superior of both sides as $l\to\infty$, we have
\begin{align*}
    \underset{l\to\infty}{\lim\sup}\; \|\mathcal{S}^{k_l}\|_{2,\psi} \leq \|\bar{\mathcal{S}}\|_{2,\psi}.
\end{align*}
Recall the definition of (limiting) subdifferential. Since we have $\mathcal{S}^{k_l}\to\bar{\mathcal{S}}$, $\|\mathcal{S}^{k_l}\|_{2,\psi} \to \|\bar{\mathcal{S}}\|_{2,\psi}$, $A_{\mathcal{S}}^{k_l}-\nabla_{\mathcal{S}} H(\mathcal{Z}^{k_l},E^{k_l},\mathcal{S}^{k_l})\in \tau \partial \|\cdot\|_{2,\psi}(\mathcal{S}^{k_l})$ and $A_{\mathcal{S}}^{k_l}-\nabla_{\mathcal{S}} H(\mathcal{Z}^{k_l},E^{k_l},\mathcal{S}^{k_l})\to \nabla_{\mathcal{S}} H(\bar{\mathcal{Z}},\bar{E},\bar{\mathcal{S}})$, as $l\to\infty$, we have~\eqref{eq:gradXH1} is satisfied. 

Similarly, we can show \eqref{eq:gradXH3} is satisfied.
\end{proof}

\begin{remark} The parameter $\alpha_E$ can be chosen as $0$, only if $(\mathcal{O}-\bar{\mathcal{S}})_{(3)}(\bar{\mathcal{Z}}_{(3)})^{\top}$ is full-rank. In this case, we can achieve $\displaystyle\lim_{k\to\infty} \|E^k-E^{k-1} \|_{F}=0$ using~\eqref{eq:Eupdate}, $\displaystyle\lim_{k\to\infty} \|\mathcal{Z}^k-\mathcal{Z}^{k-1} \|_{F}=0$, and $\displaystyle\lim_{k\to\infty} \|\mathcal{S}^k-\mathcal{S}^{k-1} \|_{F}=0$, without using  $\alpha_E>0$.
\end{remark}


\section{Numerical Experiments}\label{sec:experiments}

In this section, we conduct numerical experiments for anomaly detection in noisy HSIs. We compare the proposed method with RX~\cite{reed1990adaptive}, LRASR~\cite{xu2015anomaly}, LSCTV~\cite{feng2021local}, TLRSR~\cite{wang2022learning}, Auto-AD~\cite{wang2021auto} methods. All the numerical experiments are executed on a personal desktop with an Intel Core i7 9750H at 2.60 GHz with 16 GB RAM.  

We test the proposed anomaly detection method and the competing methods on four HSIs from the ``airport-beach-urban" dataset~\cite{kang2017hyperspectral}. In this dataset, items such as boats, cars, and airplanes are labeled as anomalous objects. The details of the test HSIs are shown as follows:

\begin{itemize}
    \item ``Airport" was captured by the airborne visible/infrared imaging spectrometer (AVIRIS)~\cite{green1998imaging} sensor over Los Angeles, with a spatial resolution of $7.1$~m. The HSI size is $100 \times 100 \times 205$. 

    \item ``Beach" was captured by the ROSIS-03 sensor in Pavia, with a spatial resolution of $1.3$~m. The HSI size is $150 \times 150 \times 102$. 

    \item ``Urban 1" and ``Urban 2" were captured by the AVIRIS over Los Angeles, with a spatial resolution of $7.1$~m. The HSIs are of size $100 \times 100 \times 205$. 

\end{itemize}
We test all the methods for anomaly detection on the original HSIs and on simulated noisy HSIs degraded by Gaussian noise with a noise level of $0.03$.

\subsection{Comparison of experimental results}

To evaluate the accuracy of object detection, we plot the receiver operating characteristic (ROC) curve~\cite{kerekes2008receiver} and calculate the area under the ROC curve (AUC)~\cite{khazai2011anomaly}. The ROC curve plots the probability of detection vs the false alarm rate for various possible thresholds. The closer the ROC curve is to the upper-left corner and the larger AUC score, the better the detection performance. 

The average AUC scores and average computational time in seconds are presented in Table~\ref{table:AUC}. The ROC curves obtained from HSIs with no noise and with a noise level of $0.03$ are presented in Figure~\ref{fig:ROC1} and Figure~\ref{fig:ROC2}, respectively. In terms of computational time, as shown in Table~\ref{table:AUC}, the proposed PnP-PBCD method completes detection in an average of  less than 15 seconds, although the traditional RX method remains the fastest overall. Regarding detection accuracy, as also shown in Table~\ref{table:AUC}, when no noise is presented, the proposed PnP-PBCD method achieves the highest average AUC score. The other methods have lower average AUC scores due to poor performance in one or more scenes, as shown in Figure~\ref{fig:ROC1}. This issue becomes more visible when the observed HSI is degraded by noise. For example, the LRASR method performs competitively with our method in ``Urban 1" and ``Urban 2" as shown in Figure~\ref{fig:ROC2} (c) and (d). However, it performs poorly for ``Airport" and ``Beach" as shown in Figure~\ref{fig:ROC2} (a) and (b). Hence, the proposed PnP-PBCD method is more robust in anomaly detection across different test HSIs, even with noise.
\setlength{\tabcolsep}{4pt}
\begin{table}[h]
\centering
\caption{Comparison of average AUC scores and average computational time obtained by different methods.}
\begin{tabular}{c|c|cccccc} \hline \textbf{Noise} & \multirow{2}{*}{\textbf{Index}}   &  \multirow{2}{*}{\textbf{RX}}     &  \multirow{2}{*}{\textbf{LRASR}}    &  \multirow{2}{*}{\textbf{LSCTV}}   &  \multirow{2}{*}{\scalebox{0.9}[1]{\textbf{Auto-AD}}}    &  \multirow{2}{*}{\textbf{TLRSR}} & \scalebox{0.9}[1]{\textbf{PnP-PBCD}}\\ 
\textbf{level}&&&&&&& \textbf{(ours})\\\hline 
\multirow{2}{*}{$0$} 
&  Ave AUC & 0.9550 & 0.8702 & 0.9137 & 0.9437 & 0.9500 & {\bf 0.9663} \\
&   Ave Time  & {\bf 0.21} & 4.06 & 364.13 & 3.75 & 4.79 & 14.86\\\hline
\multirow{2}{*}{$0.03$}  
&   Ave AUC   & 0.8948 & 0.8831 & 0.9089 & 0.9478 & 0.7939 & {\bf 0.9607}  \\
&  Ave Time & {\bf 0.21} & 5.02 & 382.90 & 4.47 & 5.15 & 13.62 \\\hline
\end{tabular}
\label{table:AUC}
\end{table}

\begin{figure}[h]
\begin{minipage}{0.82\linewidth}
\begin{subfigure}[t]{0.48\linewidth}
    \includegraphics[width=\linewidth,trim={35 0 45 25},clip]{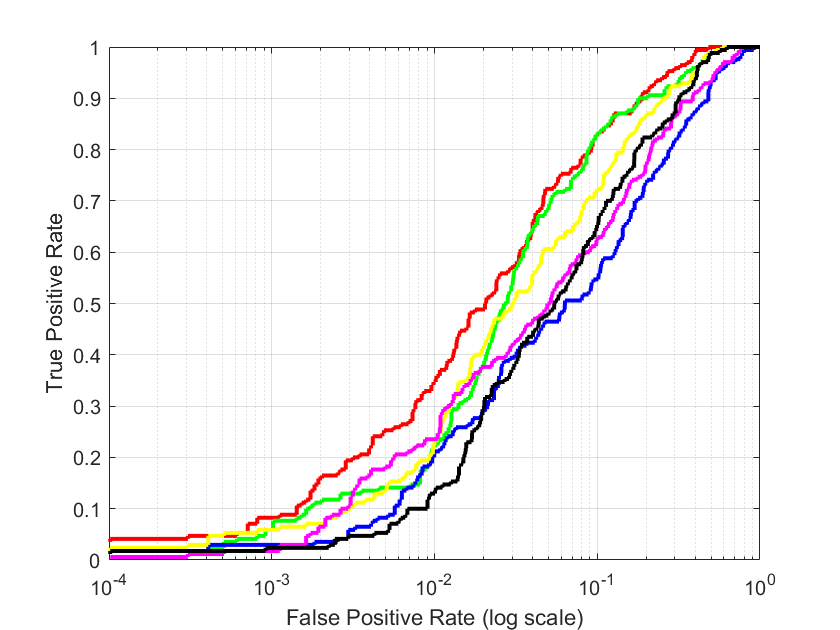}
    \caption{Airport}
    \end{subfigure}	
\begin{subfigure}[t]{0.48\linewidth}
    \includegraphics[width=\linewidth,trim={35 0 45 25},clip]{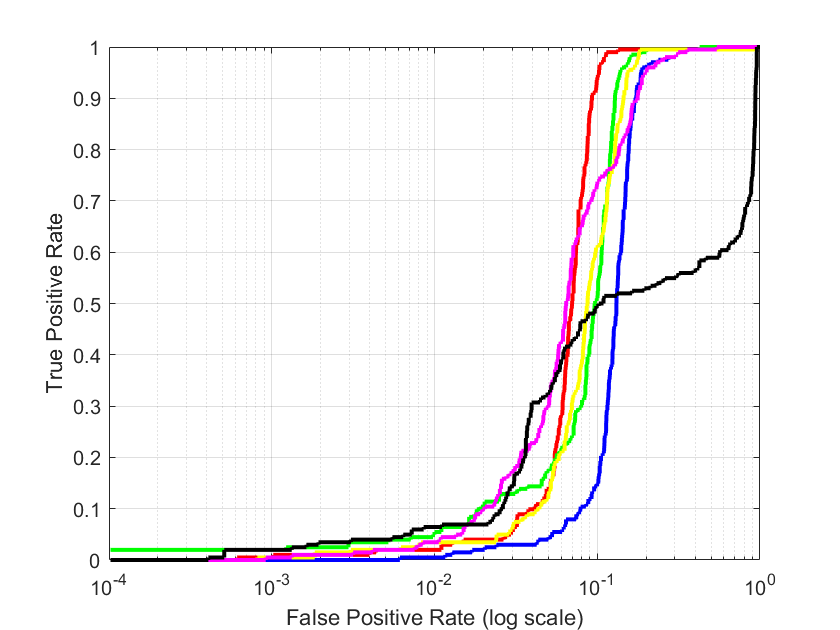}
    \caption{Beach}
    \end{subfigure}	
\end{minipage}	
\begin{minipage}[t]{0.15\linewidth}
    \begin{subfigure}[t]{\linewidth}
    \includegraphics[width=\linewidth]
    {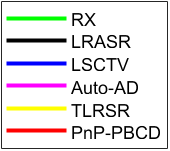}
    \end{subfigure}	
    \end{minipage}	

\begin{flushleft}\begin{minipage}{0.82\linewidth}
\begin{subfigure}[t]{0.48\linewidth}
    \includegraphics[width=\linewidth,trim={35 0 45 25},clip]{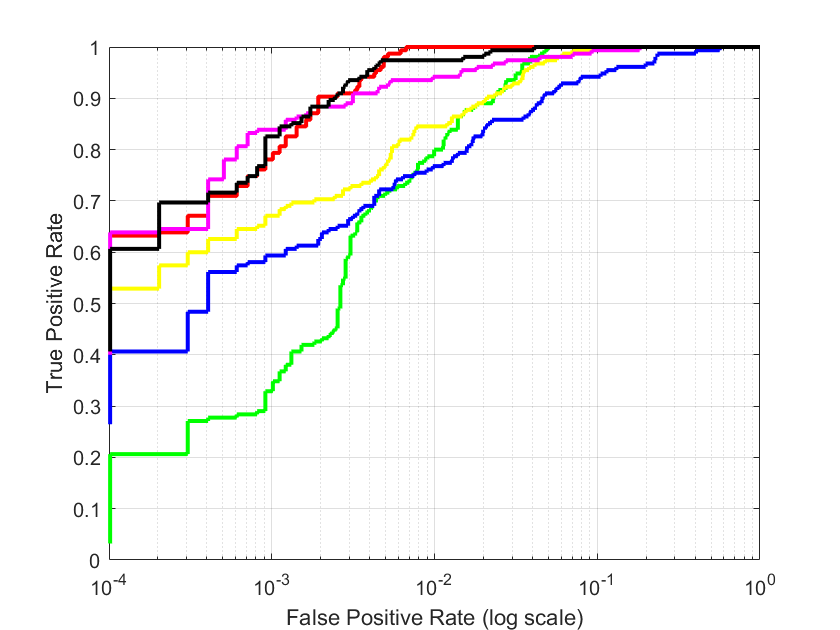}
    \caption{Urban 1}
    \end{subfigure}	
    \begin{subfigure}[t]{0.48\linewidth}
    \includegraphics[width=\linewidth,trim={35 0 45 25},clip]{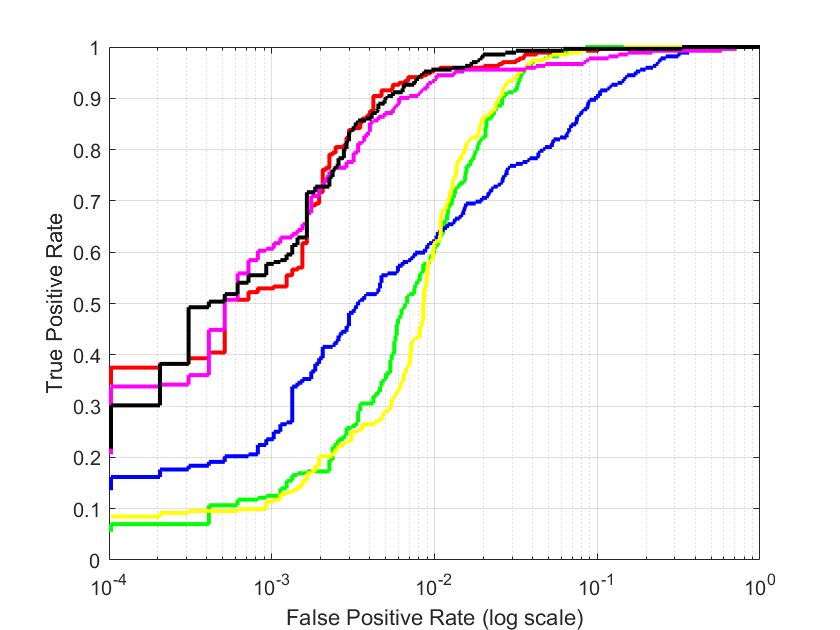}
    \caption{Urban 2}
    \end{subfigure}	
    \end{minipage}	
    \end{flushleft}
\caption{Comparison of ROC curves obtained by different methods from HSIs with no noise.}\label{fig:ROC1}
\end{figure}

\begin{figure}[h]
\begin{minipage}{0.82\linewidth}
\begin{subfigure}[t]{0.48\linewidth}
    \includegraphics[width=\linewidth,trim={35 0 45 25},clip]{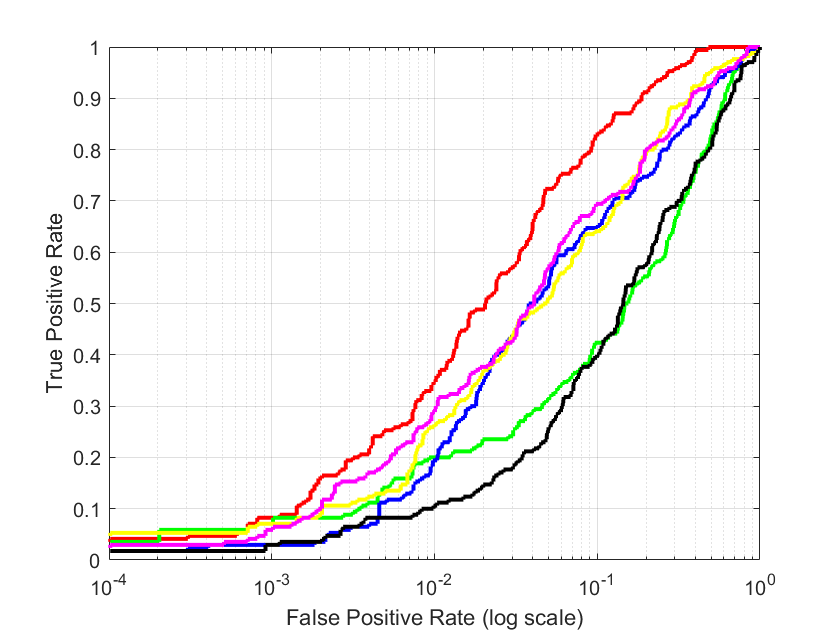}
    \caption{Airport}
    \end{subfigure}	
\begin{subfigure}[t]{0.48\linewidth}
    \includegraphics[width=\linewidth,trim={35 0 45 25},clip]{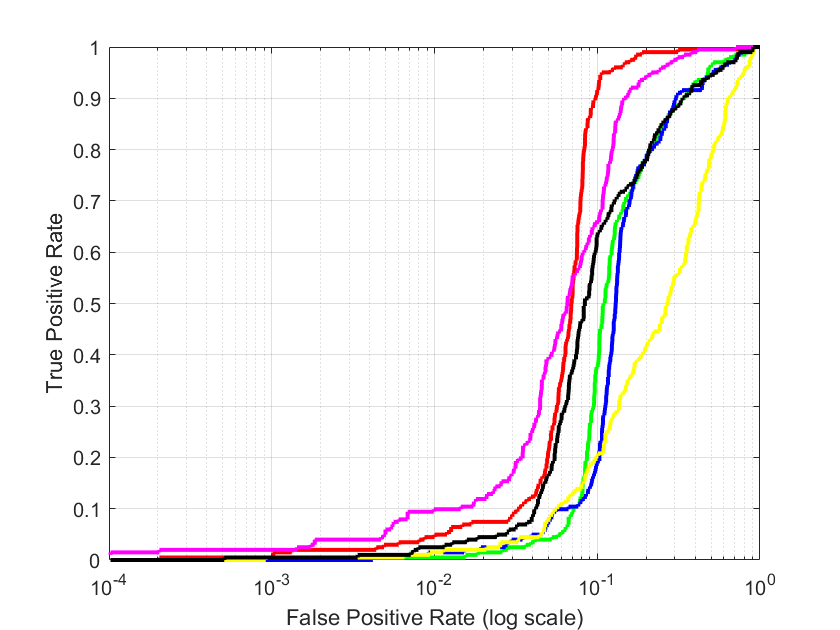}
    \caption{Beach}
    \end{subfigure}	
    \end{minipage}	
\begin{minipage}[t]{0.15\linewidth}
    \begin{subfigure}[t]{\linewidth}
    \includegraphics[width=\linewidth]
    {image/ROC/Legend2.PNG}
    \end{subfigure}	
    \end{minipage}	
    
\begin{flushleft}\begin{minipage}{0.82\linewidth}
\begin{subfigure}[t]{0.48\linewidth}
    \includegraphics[width=\linewidth,trim={35 0 45 25},clip]{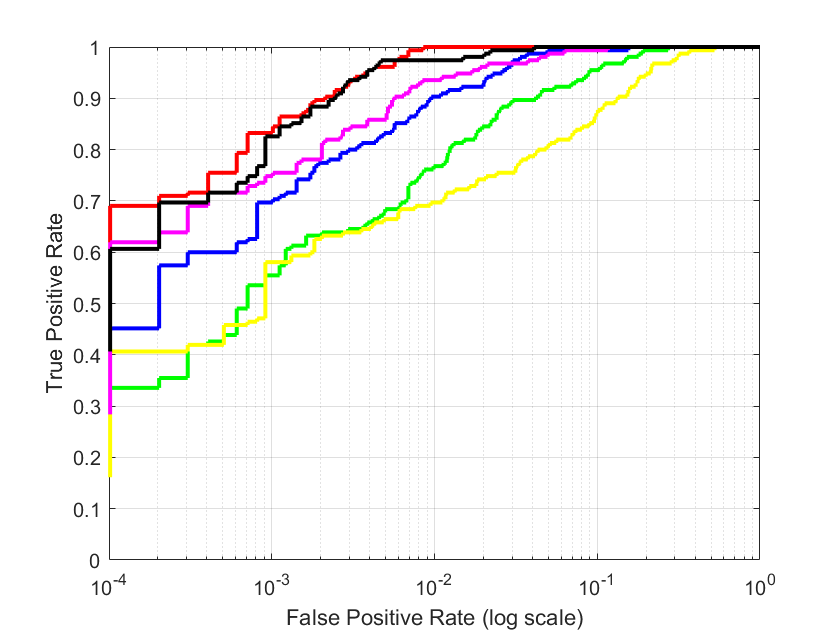}
    \caption{Urban 1}
    \end{subfigure}	
\begin{subfigure}[t]{0.48\linewidth}
    \includegraphics[width=\linewidth,trim={35 0 45 25},clip]{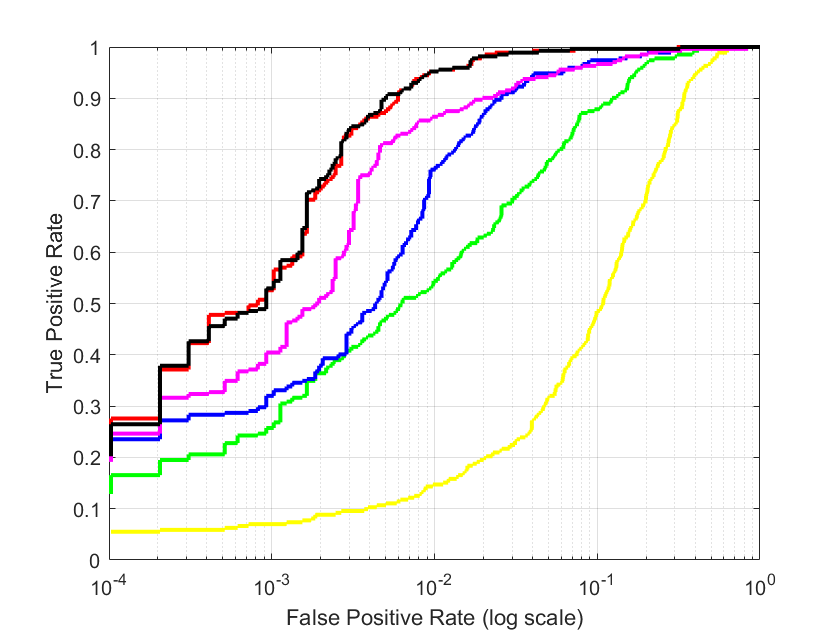}
    \caption{Urban 2}
    \end{subfigure}	
    \end{minipage}	
     \end{flushleft}
	\caption{Comparison of ROC curves obtained by different methods from HSIs with a noise level of $0.03$.}\label{fig:ROC2}
\end{figure}

For visual quality comparison, the anomalous objects detected by different methods are presented in Figures \ref{fig:Airport} and \ref{fig:Urban1} for observed HSIs with no noise. We observe that the LRASR, LSCTV, TLRSR, and proposed PnP-PBCD methods can detect rich patterns in complex scenes, as shown in Figure \ref{fig:Airport}, while the Auto-AD and proposed PnP-PBCD methods can ignore unnecessary objects in simpler scenes, as shown in Figure \ref{fig:Urban1}. Next, we present the anomalous objects detected by different methods in Figures \ref{fig:Beach} and \ref{fig:Urban2} for observed HSIs with a noise level of $0.03$. The noise significantly reduces the detection quality of the RX and TLRSR methods, as these methods also recognize noise as anomalies. Our proposed methods remain stable even with noise. Also, compared with the Auto-AD method, the proposed PnP-PBCD method achieves more complete detection and stronger density in Figures \ref{fig:Beach} and \ref{fig:Urban2}, indicating greater confidence in detecting anomalous objects.

All in all, the proposed PnP-PBCD method is an efficient method for anomaly detection, outperforming the competing methods in terms of detection accuracy and visual quality, especially in noisy HSIs.

\begin{figure}[h]
	\centering
	\begin{subfigure}[t]{0.242\linewidth}
    \includegraphics[width=\linewidth]{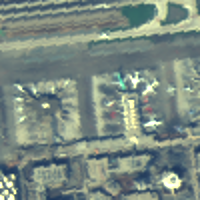}
    \caption{Observed HSI}
    \end{subfigure}	
    \begin{subfigure}[t]{0.242\linewidth}
    \includegraphics[width=\linewidth]{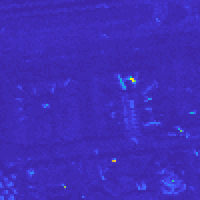}
    \caption{RX}
    \end{subfigure}
    \begin{subfigure}[t]{0.242\linewidth}
    \includegraphics[width=\linewidth]{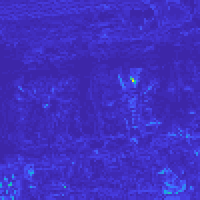}
    \caption{LRASR}
    \end{subfigure}
    \begin{subfigure}[t]{0.242\linewidth}
    \includegraphics[width=\linewidth]{image/Anomaly/abu-airport-3/000/abu-airport-3_LSCTV.png}
    \caption{LSCTV}
    \end{subfigure}
    \begin{subfigure}[t]{0.242\linewidth}
    \includegraphics[width=\linewidth]{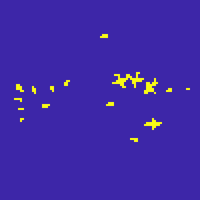}
    \caption{Ground truth}
    \end{subfigure}	
     \begin{subfigure}[t]{0.242\linewidth}
    \includegraphics[width=\linewidth]{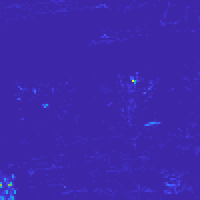}
    \caption{Auto-AD}
    \end{subfigure}
    \begin{subfigure}[t]{0.242\linewidth}
    \includegraphics[width=\linewidth]{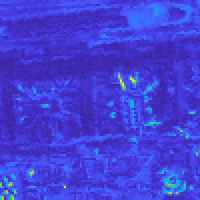}
    \caption{TLRSR}
    \end{subfigure}
    \begin{subfigure}[t]{0.242\linewidth}
    \includegraphics[width=\linewidth]{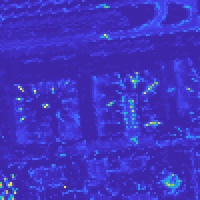}
    \caption{{\footnotesize PnP-PBCD (ours)}}
    \end{subfigure}
	\caption{Comparison of anomalous objects detected by different methods from ``Airport" with no noise.}\label{fig:Airport}
\end{figure}

\begin{figure}[h]
	\centering
	\begin{subfigure}[t]{0.242\linewidth}
    \includegraphics[width=\linewidth]{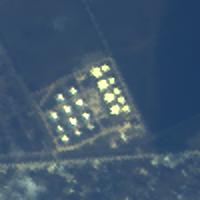}
    \caption{Observed HSI}
    \end{subfigure}	
    \begin{subfigure}[t]{0.242\linewidth}
    \includegraphics[width=\linewidth]{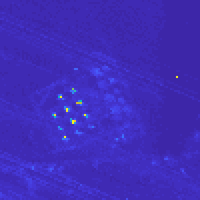}
    \caption{RX}
    \end{subfigure}
    \begin{subfigure}[t]{0.242\linewidth}
    \includegraphics[width=\linewidth]{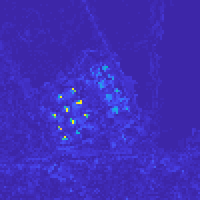}
    \caption{LRASR}
    \end{subfigure}
    \begin{subfigure}[t]{0.242\linewidth}
    \includegraphics[width=\linewidth]{image/Anomaly/abu-urban-2/000/abu-urban-2_LSCTV.png}
    \caption{LSCTV}
    \end{subfigure}
    \begin{subfigure}[t]{0.242\linewidth}
    \includegraphics[width=\linewidth]{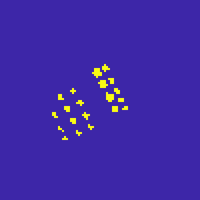}
    \caption{Ground truth}
    \end{subfigure}	
     \begin{subfigure}[t]{0.242\linewidth}
    \includegraphics[width=\linewidth]{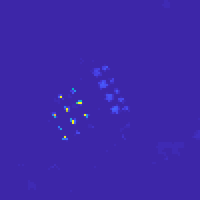}
    \caption{Auto-AD}
    \end{subfigure}
    \begin{subfigure}[t]{0.242\linewidth}
    \includegraphics[width=\linewidth]{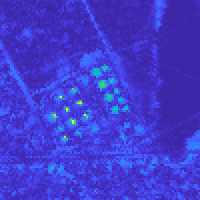}
    \caption{TLRSR}
    \end{subfigure}
    \begin{subfigure}[t]{0.242\linewidth}
    \includegraphics[width=\linewidth]{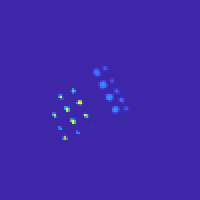}
    \caption{{\footnotesize PnP-PBCD (ours)}}
    \end{subfigure}
	\caption{Comparison of anomalous objects detected by different methods from ``Urban 1" with no noise.}\label{fig:Urban1}
\end{figure}

\begin{figure}[h]
	\centering
	\begin{subfigure}[t]{0.242\linewidth}
    \includegraphics[width=\linewidth]{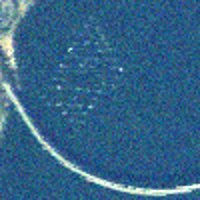}
    \caption{Observed HSI}
    \end{subfigure}	
    \begin{subfigure}[t]{0.242\linewidth}
    \includegraphics[width=\linewidth]{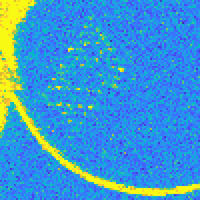}
    \caption{RX}
    \end{subfigure}
    \begin{subfigure}[t]{0.242\linewidth}
    \includegraphics[width=\linewidth]{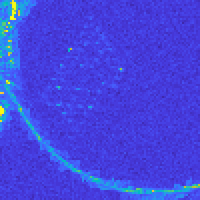}
    \caption{LRASR}
    \end{subfigure}
    \begin{subfigure}[t]{0.242\linewidth}
    \includegraphics[width=\linewidth]{image/Anomaly/abu-beach-2/003/abu-beach-2_LSCTV.png}
    \caption{LSCTV}
    \end{subfigure}
    \begin{subfigure}[t]{0.242\linewidth}
    \includegraphics[width=\linewidth]{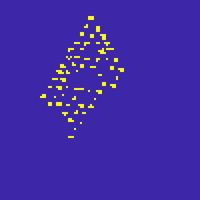}
    \caption{Ground truth}
    \end{subfigure}	
     \begin{subfigure}[t]{0.242\linewidth}
    \includegraphics[width=\linewidth]{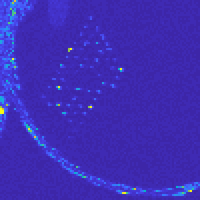}
    \caption{Auto-AD}
    \end{subfigure}
    \begin{subfigure}[t]{0.242\linewidth}
    \includegraphics[width=\linewidth]{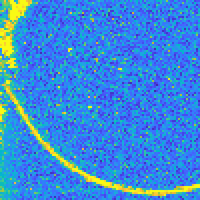}
    \caption{TLRSR}
    \end{subfigure}
    \begin{subfigure}[t]{0.242\linewidth}
    \includegraphics[width=\linewidth]{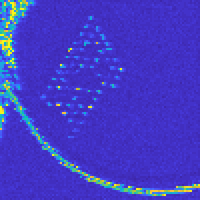}
    \caption{{\footnotesize PnP-PBCD (ours)}}
    \end{subfigure}
	\caption{Comparison of anomalous objects detected by different methods from ``Beach" with a noise level of $0.03$.}\label{fig:Beach}
\end{figure}

\begin{figure}[h]
	\centering
	\begin{subfigure}[t]{0.242\linewidth}
    \includegraphics[width=\linewidth]{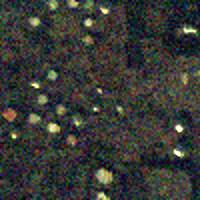}
    \caption{Observed HSI}
    \end{subfigure}	
    \begin{subfigure}[t]{0.242\linewidth}
    \includegraphics[width=\linewidth]{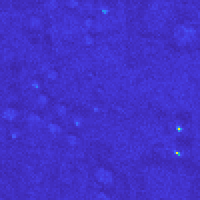}
    \caption{RX}
    \end{subfigure}
    \begin{subfigure}[t]{0.242\linewidth}
    \includegraphics[width=\linewidth]{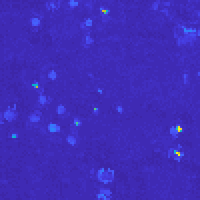}
    \caption{LRASR}
    \end{subfigure}
    \begin{subfigure}[t]{0.242\linewidth}
    \includegraphics[width=\linewidth]{image/Anomaly/abu-urban-4/003/abu-urban-4_LSCTV.png}
    \caption{LSCTV}
    \end{subfigure}
    \begin{subfigure}[t]{0.242\linewidth}
    \includegraphics[width=\linewidth]{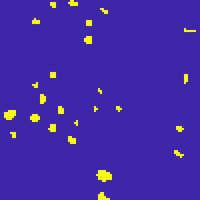}
    \caption{Ground truth}
    \end{subfigure}	
     \begin{subfigure}[t]{0.242\linewidth}
    \includegraphics[width=\linewidth]{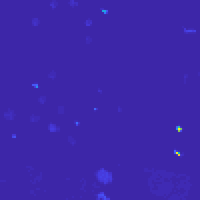}
    \caption{Auto-AD}
    \end{subfigure}
    \begin{subfigure}[t]{0.242\linewidth}
    \includegraphics[width=\linewidth]{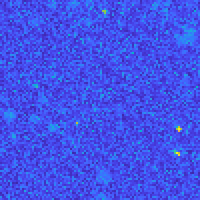}
    \caption{TLRSR}
    \end{subfigure}
    \begin{subfigure}[t]{0.242\linewidth}
    \includegraphics[width=\linewidth]{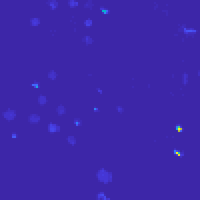}
    \caption{{\footnotesize PnP-PBCD (ours)}}
    \end{subfigure}
	\caption{Comparison of anomalous objects detected by different methods from ``Urban 2" with a noise level of $0.03$.}\label{fig:Urban2}
\end{figure}

\subsection{Parameter settings and analysis} 

For the proposed PnP-PBCD method, we set model parameters $\delta=0.25$, and $\tau =1$; and we select the relaxed $\ell_p$ norm with $p=0.1$ and $\varepsilon=10^{-5}$ as $\psi$ for $\|\cdot\|_{2,\psi}$ given in~\eqref{eq:L2pNorm}. We also set the parameters $a=0.2$, $b=0.4$, and $\gamma=0.99$ for the shifted and relaxed denoiser $\widetilde{\mathcal{D}}_{\sigma}^{\gamma}$ given in~\eqref{eq:denoiser}; and we set algorithm parameters  $\alpha_{\mathcal{S}}=\alpha_{E}=\alpha_{\mathcal{Z}}=0.01$. Algorithm~\ref{alg1} stops if $\|\mathcal{S}^{k+1}-\mathcal{S}^k\|_F/\|\mathcal{S}^{k}\|_F \leq 10^{-3}$.
	
Next, we conduct a parameter analysis on the model parameters $\delta$ and $\tau$. We use Urban 1 with a noise level of $0.03$ for testing. The plot of AUC values vs parameters $\delta$ and $\tau$ is presented in Fig.~\ref{fig:parameter}. The AUC values are similar when the ratio of $\delta$ to $\tau$ (i.e., $\frac{\delta}{\tau}$) is fixed, and the AUC achieves the best performance when this ratio is approximately $0.25$.

\begin{figure}[htbp]
	\centering
		\includegraphics[height=.35\linewidth]{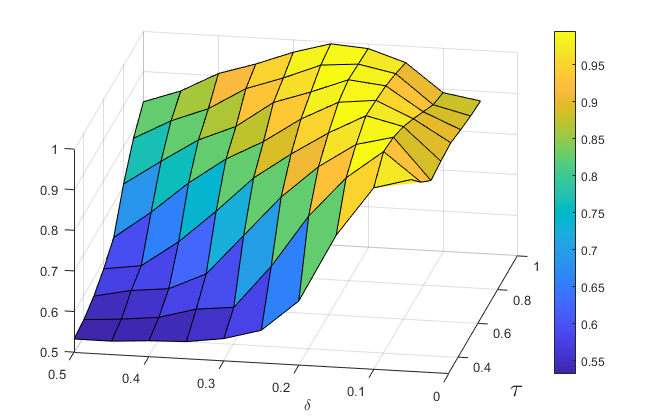}
	\caption{Plot of the AUC values vs parameters $\delta$ and $\gamma$.}\label{fig:parameter}
\end{figure}


\section{Conclusions}\label{sec:conclusions}

This paper presents a novel approach for hyperspectral anomaly detection by integrating a low-rank representation model with a deep learning-based denoiser within a PnP framework. Our method effectively addresses the challenges of noise contamination in hyperspectral images by employing a subspace representation for the background and utilizing a deep implicit prior to denoise the representation coefficients. The introduction of a generalized group sparsity measure, $\|\cdot\|_{2,\psi}$, enhances the detection of sparse anomalous objects. We developed a PnP-PBCD method to solve the resulting nonconvex nonsmooth optimization problem, ensuring that any accumulation point is a stationary point. Our experimental results demonstrate that the proposed PnP-PBCD method significantly outperforms existing state-of-the-art techniques, effectively detecting anomalous objects even in noisy conditions. 

\nocite{liu2024orthogonal,Gao2018A}

\section*{Acknowledgments}
This work was partially supported by the PolyU internal grant P0040271.

\bibliographystyle{siam}
\bibliography{bibfile}

\end{document}